\documentclass[letter, 11pt]{article}
\usepackage{amsmath,amssymb,amsthm,xypic} % insert showlabels here if needed
\input xy
\xyoption{all}

\usepackage{hyperref}

\usepackage[totalheight=24 true cm, totalwidth=15 true cm]{geometry}
\usepackage{color}

\title{On the dimension of systems of algebraic difference equations}
\author{Michael Wibmer\thanks{This work was supported by the NSF grants DMS-1760212, DMS-1760413, DMS-1760448 and the Lise Meitner grant M 2582-N32 of the Austrian Science Fund FWF.}}

\newtheorem{theo}{Theorem}[section]
\newtheorem{lemma}[theo]{Lemma}
\newtheorem{prop}[theo]{Proposition}
\newtheorem{cor}[theo]{Corollary}
\newtheorem{defi}[theo]{Definition}
\newtheorem{rem}[theo]{Remark}

\theoremstyle{definition}
\newtheorem{ex}[theo]{Example}

\newcommand{\p}{\mathfrak{p}}
\newcommand{\q}{\mathfrak{q}}

\newcommand{\Z}{\mathbb{Z}}
\newcommand{\R}{\mathbb{R}}
\newcommand{\N}{\mathbb{N}}

\newcommand{\Sol}{\operatorname{Sol}}

\newcommand{\V}{\mathbb{V}}

\newcommand{\id}{\operatorname{id}}

\newcommand{\A}{\mathbb{A}}

\renewcommand{\c}{\operatorname{c}}
\renewcommand{\d}{\operatorname{d}}
\renewcommand{\P}{\mathbb{P}}

\newcommand{\s}{\sigma}

\newcommand{\sdim}{\sigma\text{-}\dim}
\newcommand{\sdeg}{\sigma\text{-}\deg}

\newcommand{\ks}{$k$-$\s$}
\newcommand{\ord}{\operatorname{ord}}
\newcommand{\lm}{\operatorname{lm}}
\newcommand{\kb}{\overline{k}}

\begin{document}

\maketitle

%12H10 Difference algebra
%39A05 Difference equations and functional equations, general theory
%39A10  Difference equations and functional equations, additive

\begin{abstract}
	\let\thefootnote\relax\footnotetext{
		
		Michael Wibmer, Institute of Analysis and Number Theory, Graz University of Technology, Kopernikusgasse 24, 8010 Graz, Austria, \texttt{wibmer@math.tugraz.at}}
We introduce a notion of dimension for the solution set of a system of algebraic difference equations that measures the degrees of freedom when determining a solution in the ring of sequences. This number need not be an integer, but, as we show, it satisfies properties suitable for a notion of dimension. We also show that the dimension of a difference monomial is given by the covering density of its set of exponents.
\end{abstract}

\section*{Introduction}

In the algebraic theory of difference equations there has long been a focus on fields, but in the last decade the importance of studying solutions of systems of algebraic difference equations in more general difference rings has more and more been recognized. See e.g., \cite{SingerPut:difference,Hrushovski:elementarytheoryoffrobenius,Tomasic:ATwistedTheoremOfChebotarev,Tomasic:TwistedGaloisStratification,MoosaScanlon:GeneralizedHasseSchmidtVarietiesAndTheirJetSpaces,DiVizioHardouinWibmer:DifferenceGaloisofDifferential, Wibmer:FinitenessPropertiesOfAffineDifferenceAlgebraicGroups, Tomasic:AToposTheoreticViewOfDifferenceAlgebra}. In particular, the solution sets of systems of algebraic difference equations in the ring of sequences, which are of utmost importance from an applied perspective, have been studied in  
\cite{OvchinnikovPogudinScanlon:EffectiveDifferenceElimination} and \cite{PogudinScanlonWibmer:SolvingDifferenceEquationsInSequences}. 
Classical difference algebra (\cite{Cohn:difference, Levin}) provides a notion of dimension for a system of algebraic difference equations via the difference transcendence degree of an extension of difference fields. However, this approach is wholly inadequate for measuring the size of the solution set in the ring of sequences.

In respect to a system $F$ of algebraic difference equations, this shortcoming can be explained via difference ideals and difference Nullstellens\"{a}tze. In terms of difference ideals, the solution set of $F$ in difference fields corresponds to $\{F\}$, the smallest perfect difference ideal containing $F$, while the solution set of $F$ in the ring of sequences, corresponds to $\sqrt{[F]}$, the smallest radical difference ideal containing $F$. One has $\sqrt{[F]}\subseteq \{F\}$ but often this inclusion is strict. Classical difference algebra assigns a dimension to $\{F\}$, it does not provide a sensible notion of dimension for $\sqrt{[F]}$. 

Let us illustrate the situation with the concrete example $F=\{y\s(y),\ yz-z\s(z)\}$. In a difference field, i.e., in a field equipped with an endomorphism $\s$, the equation $y\s(y)=0$ implies $y=0$. But then the second equation $yz-z\s(z)=0$ implies that also $z=0$. Thus, in difference fields, the only solution of $F$ is $(y,z)=(0,0)$ and the corresponding difference dimension is $0$. On the other hand, $F$ has plenty solutions in the ring of sequences. Rewriting the system in sequence notation we obtain 
\begin{equation} \label{eq: ex intro}
y_iy_{i+1}=0, \quad y_iz_{i}-z_iz_{i+1}=0\quad \forall\ i\geq 0.
\end{equation}
For an arbitrary choice of $y_0,y_2,\ldots\in\mathbb{C}$ and $z_1,z_3,\ldots\in\mathbb{C}$ we have a sequence solution
$$ 
\begin{pmatrix} y \\
z
\end{pmatrix}
=\begin{pmatrix} 
y_0 & 0 & y_2 & 0 & \ldots \\
0 & z_1 & 0 & z_3 & \ldots
\end{pmatrix}\in (\mathbb{C}^\N)^2.
$$
According to our definition, the difference dimension of $F$ is $1$ and this number is obtained by counting the degrees of freedom when determining a solution to (\ref{eq: ex intro}): For $i\geq 0$, the maximal number of values of $y_0,y_1,\ldots,y_i,z_0,z_1,\ldots,z_i$ that can be chosen freely in a solution $(y,z)$ of (\ref{eq: ex intro}) is $i+1$. Being able to choose all of these $2(i+1)$ values freely should correspond to difference dimension $2$, thus being able to choose $i+1$ values freely corresponds to difference dimension $1$.

For a general system $F$ of algebraic difference equations, our definition of the difference dimension of $F$ is
$$\sdim(F)=\lim_{i\to\infty}\frac{d_i}{i+1},$$
where $d_i$ is the number of degrees of freedom available when determining a sequence solution of $F$ up to order $i$. Implicit in the above definition is the important and non-trivial fact that this limit exists.

The above definition of the difference dimension can be seen as an algebraic version of the mean dimension, an important numerical invariant of discrete dynamical systems first introduced by M. Gromov in \cite{Gromov:TopologicalInvariantsOfDynamicalSystemsAndSpacesOfHolomorphicMapsI}.	
Our definition is also in line with the description of the transformal dimension given by E. Hrushovski in \cite[Section 4.1]{Hrushovski:elementarytheoryoffrobenius}: ``If one thinks of sequences $(a_i)$ with $\s(a_i) = a_{i+1}$, the transformal dimension measures, intuitively, the
eventual number of degrees of freedom in choosing $a_{i+1}$, given the previous elements of the
sequence.''

In case $F$ is a perfect difference ideal, the above definition agrees with the standard definition via the difference transcendence degree. Thus, our definition of the difference dimension provides a meaningful generalization of the standard definition to situations where the approach via the difference transcendence degree cannot be applied. 

For a system $F$ of algebraic difference equations in $n$ difference variables, the difference dimension of $F$ takes a value between $0$ and $n$. However, it does not need to be an integer. For example, the difference dimension of the difference monomial $y\s(y)\ldots\s^m(y)$ is $\frac{m}{m+1}$. This corresponds to the fact that when determining a solution to $y_iy_{i+1}\ldots y_{i+m}=0,\ i\geq 0$, in essence, every $(m+1)$-st entry of $y$ has to be zero, whereas all the other entries can be chosen freely. It is non-trivial to determine the difference dimension of a general univariate difference monomial. In fact, we will show that the difference dimension of $\s^{\alpha_1}(y)^{\beta_1}\ldots\s^{\alpha_m}(y)^{\beta_m}$ equals $1-c(\{\alpha_1,\ldots,\alpha_m\})$, where $c(\{\alpha_1,\ldots,\alpha_m\})$ denotes the \emph{covering density} of $\{\alpha_1,\ldots,\alpha_m\}\subseteq \mathbb{Z}$, a classical invariant in additive number theory.

Our notion of difference dimension can very conveniently be expressed in terms of difference algebras. In fact we assign a difference dimension to an arbitrary finitely difference generated difference algebra over a difference field. Even though this number need not be an integer, we are able to show that the difference dimension of a finitely difference generated difference algebra satisfies all the properties one might expect by way of analogy with the familiar case of finitely generated algebras over a field.

As our difference dimension need not be an integer, it is natural to ask: When is it an integer and what values can occur? We isolate several cases 
in which the difference dimension is an integer. For example, we show that the difference dimension of a finitely difference generated difference algebra is an integer if the difference algebra can be equipped with the structure of a Hopf-algebra in such a way that the Hopf-algebra structure maps commute with $\s$. We do not fully answer the question which numbers occur as difference dimensions, but we reduce this question to a purely combinatorial problem. 

In this article we are only concerned with ordinary difference equations. That is, we only consider a single endomorphism $\s$. We think it would be interesting to extend the definitions and results to the more general case of several commuting endomorphisms $\s_1,\ldots,\s_n$.

\medskip

We conclude the introduction with an overview of the article. In Section \ref{sec: counting} we make precise how to count the degrees of freedom when determining sequence solutions and we define the difference dimension of a system of algebraic difference equations based on this. In Section \ref{sec: The difference dimension} we define the difference dimension of a finitely difference generated difference algebra and show that it has several nice properties, e.g., it is compatible with base change and additive over tensor products. In Section \ref{sec: Comparison} we then compare our notion of difference dimension with two other notions in the literature: The classical one defined via the difference transcendence degree and the difference Krull dimension defined via chains of prime difference ideals. In Section~\ref{sec:Covering Density and the dimension of difference monomials} we establish the connection between the difference dimension and the covering density. Finally, in the last section we discuss which numbers occur as difference dimension.

The author is grateful to Marc Technau, Lei Fu and the anonymous referees for helpful comments and suggestions.

%See \cite[Section A.7]{DiVizioHardouinWibmer:DifferenceGaloisofDifferential}, \cite[p. 308f]{Levin}, \cite{KrauseLenegan:GrowthOfAlgebrasAndGelfandKirillovDimension}.??
%

\section{Counting degrees of freedom in the ring of sequences}
\label{sec: counting}

In this section we define the difference dimension of a system of algebraic difference equations by counting the degrees of freedom encountered, when writing down a solution in the ring of sequences. The reader mainly interested in difference algebras could in principle skip this section and be content with the definition of the difference dimension of a difference algebra given in Section \ref{sec: The difference dimension}. On the other hand, the reader with a more applied background, mainly interested in solutions in the ring of sequences, might find the definition of the difference dimension given in this section much more illuminating than the more abstract approach of Section \ref{sec: The difference dimension}.

%Let $k$ be a $\s$-field and $F\subseteq k\{y_1,\ldots,y_n\}$. One can define an \emph{initial datum} as a subset $T$ of $k^{\N\times\{1,\ldots,n\}}$ together with an element $b\in k^T$.
%A solution to the \emph{initial value problem} $F(a)=0,\ a|_T=b$ is then a solution $a\in k^{\N\times\{1,\ldots,n\}}$ of $F$ such that the restriction of $a$ to $T$ agrees with $b$. Ideally, one might hope to be able to choose $T$ in such a way, that for every choice of $b\in k^T$, there exists a unique solution $a$ for the initial value problem $F(a)=0,\ a|_T=b$.
%For example, if $F$ consists of a single linear difference equation $\s^m(y)+c_{m-1}\s^{m-1}(y)+\ldots+c_0y\in k\{y\}$, the set $T=\{0,1,\ldots,m-1\}$ would satisfy this requirement. 
%
%Such a set $T$ would embody the degrees of freedom encountered, when writing down a solution to $F$ in the ring of sequences: Any solution $a$ to $F$ in $k^{\N\times\{1,\ldots,n\}}$ is determined by its restriction to $T$ and conversely, any choice of values of $a$ on $T$ determines a unique solution. Even more ambitiously, one might hope that any two such sets $T$ have the same density. Then one could define the $\s$-dimension of $F$ as the density of $T$ (normalized to take a value between $0$ and $n$, rather than between $0$ and $1$). 

\subsection{Notation}

We start by recalling some basic definitions from difference algebra (\cite{Cohn:difference,Levin}) and by fixing notation that will be used throughout the text.
All rings are assumed to be commutative and unital. $\N$ denotes the natural numbers including zero.

 A \emph{difference ring}, or \emph{$\s$-ring} for short, is a ring $R$ together with a ring endomorphism $\s\colon R\to R$. A morphism between $\s$-rings $R$ and $S$ is a morphism of rings $R\to S$ such that
$$ 
\xymatrix{
	R \ar[r] \ar_\s[d] & S \ar^\s[d] \\
	R \ar[r] & S
}
$$
commutes. In this situation $S$ is also called an \emph{$R$-$\s$-algebra}. A morphism of $R$-$\s$-algebras is a morphism of $R$-algebras that is a morphism of $\s$-rings. The tensor product $S_1\otimes_R S_2$ of two $R$-$\s$-algebras is an $R$-$\s$-algebra via $\s(s_1\otimes s_2)=\s(s_1)\otimes \s(s_2)$.

 An ideal $I$ in a $\s$-ring $R$ is a \emph{$\s$-ideal} if $\s(I)\subseteq I$. In that case $R/I$ naturally inherits the structure of a $\s$-ring such that $R\to R/I$ is a morphism of $\s$-rings.
For a subset $F$ of $R$, the smallest $\s$-ideal of $R$ containing $F$ is denoted by $[F]$, so $[F]=(F,\s(F),\ldots)$.

The \emph{$\s$-polynomial ring} $R\{y\}=R\{y_1,\ldots,y_n\}$ over a $\s$-ring $R$ in the $\s$-variables $y_1,\ldots,y_n$ is the polynomial ring over $R$ in the variables  $\s^i(y_j)$ ($i\in\N, 1\leq j\leq n$)
with action of $\s$ extended from $R$ as suggested by the names of the variables. The \emph{order} $\ord(f)$ of a $\s$-polynomial $f$ is the maximal $i$ such that $\s^i(y_j)$ occurs in $f$ for some $j$. For $f\in R\{y_1,\ldots,y_n\}$, $S$ an $R$-$\s$-algebra and $a=(a_1,\ldots,a_n)\in S^n$, the expression $f(a)$ denotes the element of $S$ obtained by substituting $\s^i(y_j)$ with $\s^i(a_j)$ in $f$.

An $R$-subalgebra of an $R$-$\s$-algebra is an $R$-$\s$-subalgebra if it is stable under $\s$. Let $S$ be an $R$-$\s$-algebra and $A\subseteq S$. The smallest $R$-$\s$-subalgebra of $S$ containing $A$ is denoted with $R\{A\}$. Explicitly, $R\{A\}=R[A,\s(A),\ldots]$. If there exists a finite subset $A$ of $S$ such that $S=R\{A\}$, then $S$ is called \emph{finitely $\s$-generated} (over $R$).

 A difference ring $R$ is a \emph{$\s$-field} if $R$ is a field.
 An $R$-$\s$-algebra $S$ with $R$ and $S$ fields is a \emph{$\s$-field extension}.

  {\bf Throughout this article $k$ will denote a $\s$-field} and $\kb$ denotes an algebraic closure of $k$. (It is possible to extend $\s$ from $k$ to $\kb$ but we have no need to choose such an extension.) The Krull-dimension of a finitely generated $k$-algebra $R$ is denoted with $\dim(R)$.

Let $Y$ be a (not necessarily finite) set of variables over $\kb$ and let $F\subseteq k[Y]$. We denote the set of solutions of $F$ in $\kb^Y$ with $\V(F)$. Affine space of dimension $n$ over $\kb$ is denoted with $\A^n=\kb^n$.

\subsection{Affine sequence solutions} \label{subsec: Affine sequence solutions}

We consider the set $\kb^{\N}$ of sequences in $\kb$ as a $\s$-ring with componentwise addition and multiplication and $\s$ given by the left-shift
$\s((a_i)_{i\in\N})=(a_{i+1})_{i\in\N}$. Moreover, we consider $\kb^\N$ as a \ks-algebra via $k\to \kb^\N,\ \lambda\mapsto (\s^i(\lambda))_{i\in\N}$. For a subset $F$ of $k\{y_1,\ldots,y_n\}$ we define the set of \emph{affine sequence solutions} of $F$ as
$$\Sol^\A(F)=\big\{a\in \big(\kb^\N\big)^n|\ f(a)=0 \ \forall\  f\in F\big\}.$$
Note that $\big(\kb^\N\big)^n$ can be identified with $(\A^n)^\N$. For $$a=\big(a_{i,j}\big)_{(i,j)\in\mathbb{N}\times\{1,\ldots,n\}}\in \big(\kb^\N\big)^n= (\A^n)^\N$$ and $f\in k\{y_1,\ldots,y_n\}$ one has $f(a)=0\in \kb^\N$ if and only if $\s^i(f)(a)=0\in \kb$ for all $i\in \N$.
Thus $$\Sol^\A(F)=\Sol^\A([F])=\V([F])\subseteq (\A^n)^\N.$$
For a finite subset $T$ of $\N\times \{1,\ldots,n\}$
% Then $T$ is uniquely of the form $$T=\biguplus_{\ell=1}^r T_{i_\ell}$$
%with $T_{i_\ell}\subseteq \{i_\ell\}\times\{1,\ldots,n\}$ non-empty for $\ell=1,\ldots, r$.
% We set
%$$k[y_T]=k\big[\s^i(y_j)|\ (i,j)\in T\big]\subseteq k\{y_1,\ldots,y_n\}$$ 
%and for $a=(a_1,\ldots,a_n)\in (\kb^\N)^n$ we set $a_T=((a_j)_i)_{(i,j)\in T}\in \kb^T=^T$.
%Then, for $f\in k[y_T]\subseteq k\{y_1,\ldots,y_m\}$ and $a\in (K^\N)^n$ one has $f(a)=0\in K^\N$ if and only if $\s^i(f)(a_{T+i})=0\in K$ for all $i\in\N$, where $T+i=\{(i'+i,j)|\ (i',j)\in T\}$. It follows that $\Sol(F)=\Sol([F])$ can be identified with   
we set $y_T=\{\s^i(y_j)|\ (i,j)\in T\}$ and  $$\Sol^\A_T(F)=\V([F]\cap k[y_T])\subseteq \A^T,$$
where $\A^T$ is an affine space of dimension $|T|$. 
The projection maps $$(\A^n)^\N\to \A^T,\ \big(a_{i,j}\big)_{(i,j)\in\mathbb{N}\times\{1,\ldots,n\}}\mapsto \big(a_{i,j}\big)_{(i,j)\in T}$$ induce maps 
$$\pi_T\colon \Sol^\A(F)\to \Sol^\A_T(F).$$

As a first approximation to counting the degrees of freedom encountered, when writing down an affine sequence solution of $F$, one may feel tempted to say that $T$ is free with respect to $F$ if every $a_T\in \A^T$ extends to an affine sequence solution of $F$, i.e., if $\pi_T(\Sol^\A(F))=\A^T$. Or, in other words, if the initial value problem $$f(a)=0\  \forall \ f\in F,\quad \pi_T(a)=a_T$$
has a solution $a\in \big(\kb^\N\big)^n$ for all $a_T\in \A^T$. However, as illustrated in the following simple example, such a definition would be too stringent.  

\begin{ex} \label{ex: infinity}
Let us consider the affine sequence solutions of the $\s$-polynomial	$f=y_1\s(y_1)-1$ over $(k,\s)=(\mathbb{C},\id)$. A sequence $a=(a_i)_{i\in \N}\in \mathbb{C}^\N$ is a solution if and only if $a_ia_{i+1}=1$. Thus 
$$\Sol^\A(f)=\{(a_0,a_0^{-1},a_0,a_0^{-1},\ldots) \ | \ a_0\in\mathbb{C}\smallsetminus\{0\}\}.$$
Intuitively, we should count one degree of freedom here because $a_0$ can be chosen more or less arbitrarily and then all the other coefficients are determined, i.e., $T=\{0\}$ should be considered to be free. However, $a_0=0$ does not extend to an affine sequence solution. 
%This shortcoming can be remedied by also allowing solutions at infinity. In fact, $a_0=0$ extends to a projective sequence solution $(0,\infty,0,\infty,\ldots)\in(\P^1)^\mathbb{N}$. 
%This idea will be formalized in the next subsection. Cf. Example \ref{ex: infty continued}.
\end{ex}

The above example also shows that in general the projection maps $\pi_T\colon \Sol^\A(F)\to \Sol^\A_T(F)$ are not surjective. Moreover, as illustrated in the following example, the image of $\pi_T$ is in general not a constructible subset of the algebraic variety $\Sol^\A_T(F)$.

%This is essentially the path we will follow but some technical adjustments are necessary since in general no $T$ as described above will exist. Let us illustrate this with an example:

\begin{ex} \label{ex: not constructible}
	We consider the system $F=\{\s(y_1)-y_1-1,\ y_1y_2-1\}$ over $(k,\s)=(\mathbb{C},\id)$, which we may rewrite more succinctly as
%	\begin{align*}
%	\s(y) & =y+1, \\
%	yz & =1
%	\end{align*}
%	in the $\s$-variables $y$ and $z$. Or more succinctly,
	\begin{align*}
	y_{1,i+1} & =y_{1,i}+1, \\
	y_{1,i}y_{2,i} & =1.
	\end{align*}
	Clearly $y_{2,i}$ is determined by $y_{1,i}$ and $y_{1,i}$ is determined by $y_{1,i-1}$, so the only freedom available when determining an affine sequence solution of $F$ is the choice of $y_{1,0}$. But not all choices of $y_{1,0}$ yield a solution. Indeed, $y_{1,0}\in\mathbb{C}$ extends to an affine sequence solution of $F$ if and only if $y_{1,0}\neq -n$, for $n\in\N$. In other words, the image of $\pi_T\colon\Sol^\A(F)\to \Sol^{\A}_T(F)$ for $T=\{(0,1)\}$ is $\mathbb{C}\smallsetminus\{-n|\ n\in \N\}$, which is not a constructible subset of $\mathbb{C}$.
\end{ex}

Even worse, as explained in the following example, the image of $\pi_T\colon \Sol^\A(F)\to \Sol^\A_T(F)$ need not be Zariski dense in $\Sol^\A_T(F)$. 
We will see in Subsection \ref{subsec: A characterization of free sets in terms of affine} that such a pathology cannot happen if $k$ is uncountable. 

\begin{ex} \label{ex: not Zariski dense}
	We will not explicitly write down such an example but rather give an abstract argument why such an example exists. Using ideas and methods from \cite{PogudinScanlonWibmer:SolvingDifferenceEquationsInSequences} it would in principle be possible to write down an explicit example but that would be extremely tedious.
	
	It is shown in \cite[Theorem 3.2]{PogudinScanlonWibmer:SolvingDifferenceEquationsInSequences}
	that there exists an integer $n\geq 1$, a finite set $F\subseteq k\{y_1,\ldots,y_n\}$ of $\s$-polynomials over $(k,\s)=(\overline{\mathbb{Q}},\id)$ and a $\s$-polynomial $g\in k\{y_1,\ldots,y_n\}$ such that $g$ vanishes on every element of $\Sol^\A(F)$ but $g\notin \sqrt{[F]}$. Let $T\subseteq \N\times\{1,\ldots,n\}$ be such that $g\in k[y_T]$. We claim that the image of $\pi_T\colon \Sol^\A(F)\to \Sol^\A_T(F)$ is not Zariski dense in $\Sol^\A_T(F)$. As $g$ vanishes on $\Sol^\A(F)$, we see that the image of $\pi_T$ is contained in $\V(g)\subseteq \A^T$. On the other hand, as $g\notin\sqrt{[F]}$, we also have $g\notin\sqrt{[F]\cap k[y_T]}$. So $g$ does not vanish on $\Sol^\A_T(F)$. We conclude $$\pi_T(\Sol^\A(F))\subseteq \V(g)\nsubseteq \Sol^\A_{T}(F).$$
	Thus $\pi_T(\Sol^\A(F))$ is not Zariski dense in $\Sol^\A_T(F)$.
	
\end{ex}

\subsection{Projective sequence solutions}

We have seen above that for a finite subset $T$ of $\N\times\{1,\ldots,n\}$, the set of elements of $\Sol^\A_T(F)$ that extends to an affine sequence solution of $F$, is in general not Zariski dense and not constructible. In this section we show that the situation can be improved by allowing projective sequence solutions instead of just affine sequence solutions: The set of all elements of $\Sol^\A_T(F)$ that extend to a projective sequence solution of $F$ contains an open Zariski dense subset of $\Sol^\A_T(F)$ (Lemma \ref{lemma: U extends}).

%We now formalize the idea from Example \ref{ex: infinity}. The main motivation for considering projective sequence solutions instead of only affine sequence solutions is that, upon passing to the projective closures, the projection maps  $\pi_T\colon \Sol^\A(F)\to \Sol^\A_T(F)$ become surjective. 
We write $\P^n=\P^n(\kb)$ for $n$-dimensional projective space over $\kb$.

\begin{rem}[Multiprojective space] \label{rem: multiprojective space}
	Let $n,r\geq 1$. The closed subsets of the algebraic $\kb$-variety $\P^{n}\times\ldots\times\P^{n}=(\P^n)^r$ are exactly the solution sets of systems of multihomogeneous polynomials 
(cf. \cite[Chapter 1, Section 5.1]{Shafarevich:BasicAlgebraicGeometry1}).	Here a polynomial $f\in \kb[y_{1,0},\ldots,y_{1,n},\ldots,y_{r,0},\ldots,y_{r,n}]$ is called \emph{multihomogeneous} of \emph{multidegree} $(d_1,\ldots,d_r)$ if $f$ is homogeneous of degree $d_i$ in the variables $y_{i,0},\ldots,y_{i,n}$ for $i=1,\ldots,r$. For a set $F$ of multihomogeneous polynomials we write $\V^h(F)$ for the closed subset of $(\P^{n})^r$ defined by $F$. We consider $\A^{n}\times\ldots\times \A^{n}=(\A^n)^r=\A^{nr}$ as an open subset of $(\P^{n})^r$ via the embedding
$$\big((a_{1,1},\ldots,a_{1,n}),\ldots,(a_{r,1},\ldots,a_{r,n})\big)\mapsto \big((1:a_{1,1}:\ldots:a_{1,n}),\ldots,(1:a_{r,1}:\ldots:a_{r,n})\big).$$ Then $(\P^{n})^r$ is the union of  $(\A^n)^r$ and the points at infinity $\V^h(y_{1,0}\ldots y_{r,0})$.

Let $f\in \kb[y_{1,1},\ldots,y_{1,n},\ldots,y_{r,1},\ldots,y_{r,n}]$  and for $i=1,\ldots,r$ let $d_i$ denote the degree of $f$ in $y_{i,1},\ldots,y_{i,n}$. The multihomogenization $f^h\in \kb[y_{1,0},\ldots,y_{1,n},\ldots,y_{r,0},\ldots,y_{r,n}]$ of $f$ is defined as $$f^h=y_{1,0}^{d_1}\ldots y_{r,0}^{d_r}f\big(\tfrac{y_{1,1}}{y_{1,0}},\ldots,\tfrac{y_{1,n}}{y_{1,0}},\ldots,\tfrac{y_{r,1}}{y_{r,0}},\ldots,\tfrac{y_{r,n}}{y_{r,0}}\big).$$
For a closed subset $X$ of $(\A^{n})^r$, the closure $\overline{X}$ of $X$ in $(\P^{n})^r$ equals $\V^h(\mathbb{I}(X)^h)$, where $\mathbb{I}(X)\subseteq \kb[y_{1,1},\ldots,y_{1,n},\ldots,y_{r,1},\ldots,y_{r,n}]$ is the defining ideal of $X$ and $\mathbb{I}(X)^h=\{f^h|\ f\in\mathbb{I}(X)\}$.
\end{rem}

Let $\N[\s]$ denote the set of polynomials in the variable $\s$ with natural number coefficients. We consider $\N[\s]$ as an abelian monoid under addition. The $\s$-polynomial ring $k\{y_0,\ldots,y_n\}$ has a natural $\N[\s]$-grading that we shall now describe. We define the \emph{$\s$-degree} of a $\s$-monomial as 
%$$\sdeg(y^{\alpha_{0,0}}_0\ldots y^{\alpha_{0,n}}_n\ldots\s^r(y_0)^{\alpha_{r,0}}\ldots\s^r(y_n)^{\alpha_{r,n}})=(\alpha_{r,0}+\ldots+\alpha_{r,n})\s^r+\ldots+\alpha_{0,0}+\ldots+\alpha_{0,n}$$

$$\sdeg\left(\prod_{i=0}^r\prod_{j=0}^n\s^i(y_j)^{\alpha_{i,j}}\right)=\sum_{i=0}^r\left(\sum_{j=0}^n\alpha_{i,j}\right)\s^i.$$
A $\s$-polynomial $f\in k\{y_0,\ldots,y_n\}$ is \emph{$\s$-homogeneous} of $\s$-degree $\sdeg(f)=d\in\N[\s]$ if all $\s$-monomials of $f$ have $\s$-degree $d$. Thus $f$ is $\s$-homogeneous if and only if $f$ is homogeneous in $\s^i(y_0),\ldots,\s^i(y_n)$ for every $i\in\N$. 
Note that every $\s$-polynomial $f\in k\{y_0,\ldots,y_n\}$ can uniquely be written as a sum of $\s$-homogeneous $\s$-polynomials.

%Let $T$ be a finite subset of $\N\times\{1,\ldots,n\}$. 
% Then $T$ is uniquely of the form $$T=\biguplus_{\ell=1}^r T_{i_\ell}$$
% with $T_{i_\ell}\subseteq \{i_\ell\}\times\{1,\ldots,n\}$ non-empty for $\ell=1,\ldots, r$. We set $T_{i_\ell}^h=\{(i_\ell, 0)\}\cup T_{i_\ell}$, $$T^h=\biguplus_{\ell=1}^r T_{i_\ell}^h,$$
% and $$k[y_{T^h}]=k[\s^i(y_j)|\ (i,j)\in T^h]\subseteq k\{y_0,\ldots,y_n\}.$$
% 
% 
%%with $0\leq m_1< \ldots< m_r$ and $T_{m_i}\subseteq \{1,\ldots,n\}$ non-empty for $\ell=1,\ldots,r$.
%
%We set $T^h=T\cup\{(i,0)|\ \exists\ j : (i,j)\in T\}$
%We set $$k[y_T^h]=k[\s^i(y_j), \s^i(y_0)|\ (i,j)\in T]\subseteq k\{y_0,\ldots,y_n\}.$$
%Note that the grading on $k[y_{T^h}]$ induced by the $\N[\s]$-grading on $k\{y_0,\ldots,y_n\}$ exactly corresponds to the multidegree as in Remark \ref{rem: multiprojective space}. Thus a set of $\s$-homogeneous $\s$-polynomials of $k\{y_0,\ldots,y_n\}$ contained in $k[y_{T^h}]$ defines a closed subset of 
%$$\P^T=\P^{T_{i_1}}\times \ldots\times\P^{T_{i_r}},$$
%where $\P^{T_{i_\ell}}$ denotes a projective space of dimension $|T_{i_\ell}|$ and with coordinates $(a_{(i,j)})_{(i,j)\in T^h_{i_\ell}}$ for $\ell=1,\ldots,r$. 
Let $f\in k\{y_1,\ldots,y_n\}$ be of order $r$, (so $f=f(y_1,\ldots,y_n,\ldots,\s^r(y_1),\ldots,\s^r(y_n))$ and for $i=0,\ldots,r$, let $d_i$ denote the degree of $f$ in the variables $\s^i(y_1),\ldots,\s^i(y_n)$. The \emph{$\s$-homogenization} $f^h\in k\{y_0,\ldots,y_n\}$ of $f$ is defined as
$$f^h=y_0^{d_0}\ldots\s^r(y_0)^{d_r}f\big(\tfrac{y_1}{y_0},\ldots,\tfrac{y_n}{y_0},\ldots,\tfrac{\s^r(y_1)}{\s^r(y_0)},\ldots,\tfrac{\s^r(y_n)}{\s^r(y_0)}\big).$$ 
For a subset $F$ of $k\{y_1,\ldots,y_n\}$ we set $F^h=\{f^h|\ f\in F\}$.
\begin{ex}
	We have $(y_1\s(y_1)-1)^h=y_1\s(y_1)-y_0\s(y_0)$
\end{ex}

For $i\in \N$, the grading on $k[y_0,\ldots,y_n,\ldots,\s^i(y_0),\ldots,\s^i(y_n)]\subseteq k\{y_0,\ldots,y_n\}$ induced by the $\N[\s]$-grading on $k\{y_0,\ldots,y_n\}$, exactly corresponds to the multidegree as in Remark~\ref{rem: multiprojective space}. Thus, a set of $\s$-homogeneous $\s$-polynomials of $k\{y_0,\ldots,y_n\}$ of order at most $i$, defines a closed subset of $(\P^n)^{i+1}$.

We note that if $f\in k\{y_0,\ldots,y_n\}$ is $\s$-homogeneous of degree $d=d_r\s^r+\ldots+d_0$ and $a=(a_0,\ldots,a_n)\in k^{n+1}$, then
$f(\lambda a)=\lambda^{d_0}\ldots\s^r(\lambda)^{d_r}f(a)$ for all $\lambda\in k$. Thus the expression $f(b)=0$ is well-defined for $b\in\P^n(k)$. On the other hand, we can also consider $f$ as a multihomogeneous polynomial in the variables $\s^i(y_j)$ (rather than as a difference polynomial) and in this context the expression $f(a)=0$ is well-defined for any $a\in \big(\P^n\big)^\N$.
%We set
%$$\Sol^\P(F)=\big\{a\in \big(\P^n)^\N|\ f(a)=0 \ \forall \ f\in [F]^h\big\}.$$

\medskip

Let, as in Subsection \ref{subsec: Affine sequence solutions}, $F$ be a subset of $k\{y_1,\ldots,y_n\}$.  The set of \emph{projective sequence solutions}                 
	 of $F$ is
$$\Sol^\P(F)=\big\{a\in \big(\P^n)^\N|\ f(a)=0 \ \forall \ f\in [F]^h\big\}.$$
For $i\in\N$ let $T_i=\{0,\ldots,i\}\times \{1,\ldots,n\}$ and
$$\Sol^\P_i(F)=\V^h(([F]\cap k[y_{T_i}])^h)\subseteq(\P^n)^{i+1}.$$
%\V^h(k[y[i]]\cap ([F]^h))
%\s(y_0)^{\alpha_{1,0}}\ldots\s(y_n)^{\alpha_{1,n}
%$$\Sol_T^\P(F)=\V^h(k[y_{T^h}]\cap ([F]^h))=\V^h(([F]\cap k[y_T])^h)\subseteq\P^T.$$
Thus, $\Sol_i^\P(F)$ is the closure of $\Sol_{T_i}^\A(F)$ in $(\P^n)^{i+1}$ (Remark \ref{rem: multiprojective space}). Since $[F]\cap k[y_{T_i}]\subseteq [F]\cap k[y_{T_{i+1}}]$, the maps $(\P^n)^{i+2}\to (\P^n)^{i+1},\ (b_0,\ldots,b_{i+1})\mapsto (b_0,\ldots,b_i)$ induce maps 
$$\pi_{i+1,i}\colon \Sol^\P_{i+1}(F)\to \Sol^\P_i(F).$$

The standard embedding $\A^n\hookrightarrow \P^n,\ (a_1,\ldots,a_n)\mapsto(1:a_1:\ldots:a_n)$ yields an inclusion $(\A^n)^\N\subseteq(\P^n)^\N$, which, in turn, induces an inclusion $\Sol^\A(F)\subseteq\Sol^\P(F)$.
 Also, the projection maps $$(\P^n)^\N\to (\P^n)^{i+1},\ (b_0,b_1,\ldots)\mapsto (b_0,\ldots,b_i)$$ induce maps 
$\pi_i\colon \Sol^\P(F)\to \Sol_i^\P(F)$. We have commutative diagrams 
\begin{align} \label{eq: commuate for Sol}
\xymatrix{
\Sol^\A(F) \ar@{^(->}[r] \ar_{\pi_{T_i}}[d] & \Sol^\P(F) \ar^{\pi_i}[d] \\
\Sol^\A_{T_i}(F) \ar@{^(->}[r] & \Sol_i^\P(F)	
}
\text{\quad  \quad and \quad \quad   }
\xymatrix{ \Sol^\A_{T_{i+1}}(F) \ar[d] \ar@{^(->}[r] &  	\Sol^\P_{i+1}(F) \ar^-{\pi_{i+1,{i}}}[d] \\
\Sol^\A_{T_i}(F)   \ar@{^(->}[r] & \Sol^\P_{i}(F).
}
\end{align}
However, note that for an arbitrary finite subset $T$ of $\N\times\{1,\ldots,n\}$, there may not be a projective version of the map $\pi_T\colon \Sol^\A(F)\to \Sol^\A_{T}(F)$, because there are no projective analogs of the coordinate projections on $\A^n$.
% extend $\pi_{T_i}\colon \Sol^\A(F)\to \Sol^\A_{T_i}(F)$.

%
%For $i\in \N$ we set
%$$T[i]=\{0,1,\ldots,i\}\times \{1,\ldots,n\}$$
%and $\Sol^\P_i(F)=\V^h(k[y$

\begin{lemma} \label{lemma: projections surjective}
	The projection maps $\pi_i\colon\Sol^\P(F)\to \Sol^\P_i(F)$ are surjective.
\end{lemma}
\begin{proof}
	Note that $b\in(\P^n)^\N$ lies in $\Sol^\P(F)$ if and only if $\pi_i(b)\in (\P^n)^{i+1}$ lies in $\Sol^\P_i(F)$ for every $i\in\N$.
	In other words, $\Sol^\P(F)$ can be identified with the inverse limit of the $\Sol^\P_i(F)$'s. It thus suffices to show that the maps $\pi_{i+1,i}\colon  \Sol^\P_{i+1}(F)\to \Sol^\P_{i}(F)$, %(b_0,\ldots,b_{i+1})\mapsto (b_0,\ldots, b_i)$ 
	are surjective. 
%	We have a commutative diagram 
%	$$
%	\xymatrix{
%	\Sol^\P_{i+1}(F) \ar^-{\pi_{i+1,{i}}}[r]& \Sol^\P_{i}(F) \\
%	\Sol^\A_{T_{i+1}}(F) \ar[r] \ar@{^(->}[u] & \Sol^\A_{T_i}(F) \ar@{^(->}[u]
%}
%	$$
%	of algebraic $\kb$-varieties. 
The inclusion $$k[y_{T_i}]/(k[y_{T_i}]\cap[F])\hookrightarrow k[y_{T_{i+1}}]/(k[y_{T_{i+1}}]\cap [F])$$ of finitely generated $k$-algebras, corresponds to a dominant morphism of affine $k$-schemes.
	Therefore, also the morphism $\Sol^\A_{T_{i+1}}(F)\to \Sol^\A_{T_i}(F)$ of affine $\kb$-varieties is dominant. 
	As $\Sol^\P_{i}(F)$ is the closure of $\Sol^\A_{T_{i}}(F)$, this and the commutativity of (\ref{eq: commuate for Sol}), implies that also $\pi_{i+1,i}$ is dominant.
Projective space is complete and so are products and closed subvarieties of complete varieties. Thus $\Sol^\P_{i+1}(F)$ is complete. Since the image of a complete variety under a morphism is closed, it follows that $\pi_{i+1,i}$ has a dense and closed image. Therefore $\pi_{i+1,i}$ is surjective.
\end{proof}

As discussed in Section \ref{subsec: Affine sequence solutions}, for a finite subset $T$ of $\N\times\{1,\ldots,n\}$, the set of elements of $\A^T$ that extend to an affine sequence solution of $F$ is not so well-behaved. In particular, it need not contain a non-empty open subset of $\Sol^\A_T(F)$. To remedy this  situation (see Lemma \ref{lemma: U extends} below), we consider the possibility of extending elements of $\A^T$ to projective sequence solutions of $F$.

\begin{defi} \label{defi: extends}
	Let $F\subseteq k\{y_1,\ldots,y_n\}$ and let $T$ be a finite subset of $\N\times\{1,\ldots,n\}$. An element $a=(a_{i,j})_{(i,j)\in T}\in \A^T$ \emph{extends to a projective sequence solution} of $F$, if there exists  $b=(b_{i,0}:\ldots: b_{i,n})_{i\in\N}\in \Sol^\P(F)\subseteq (\P^n)^\N$ such that $a_{i,j}=b_{i,j}$ for all $(i,j)\in T$ and $b_{i0}=1$ for all $i\in \N$ with $(i,j)\in T$ for some $j$. 
\end{defi}

Clearly, if $a\in \A^T$ extends to an affine sequence solution of $F$, then $a$ also extends to a projective sequence solution of $F$. On the other hand, an $a\in\A^T$ that extends to a projective sequence solution of $F$ need not extend to an affine sequence solution of $F$, since projective sequence solutions allow the possibility of $b_{i,0}=0$ as long as $(i,j)\notin T$ for all $j\in\{1,\ldots,n\}$.

\begin{lemma} \label{lemma: extends implies solution}
	If $a\in \A^T$ extends to a projective sequence solution of $F$, then $a\in\Sol^\A_T(F)$.
\end{lemma}
\begin{proof}
Assume that $a=(a_{i,j})_{(i,j)\in T}\in \A^T$ extends to a projective sequence solution of $F$ and let $f\in [F]\cap k[y_T]$. Moreover, let $b=(b_{i,0}:\ldots: b_{i,n})_{i\in\N}\in \Sol^\P(F)$ be as in Definition \ref{defi: extends}. Let $I$ be the smallest subset of $\N$ such that $T\subseteq I\times\{1,\ldots,n\}$, i.e., $I=\{i\in \N|\ \exists \ j\in\{1,\ldots,n\} \text{ such that } (i,j)\in T\}$.
Since every element of $[F]^h$ vanishes on $b$, we see that $f^h\in k[\s^i(y_j)|\ (i,j)\in I\times \{0,\ldots,n\}]$ vanishes on $((b_{i,0}:\ldots :b_{i,n}))_{i\in I}\in (\P^n)^{|I|}$. Since $f\in k[y_T]$, the polynomial $f^h$ only involves the variables $\s^i(y_0), (i\in I)$ and $\s^i(y_j), ((i,j)\in T)$. Since $a_{i,j}=b_{i,j}$ for $(i,j)\in T$ and $b_{i,0}=1$ for $i\in I$, we see that $f^h(b)=0$ implies $f(a)=0$. So  $a\in\Sol^\A_T(F)$.
\end{proof}

\begin{lemma} \label{lemma: U extends}
	Let $F\subseteq k\{y_1,\ldots,y_n\}$ and let $T$ be a finite subset of $\N\times\{1,\ldots,n\}$. Then there exists an open Zariski dense subset $U$ of $\Sol^\A_T(F)$ such that every $a\in U$ extends to a projective sequence solution of $F$. 
\end{lemma}
\begin{proof}
	Let $i\in\N$ be such that $T\subseteq T_i=\{0,\ldots,i\}\times\{1,\ldots,n\}$. The inclusion $$k[y_{T}]/(k[y_{T}]\cap[F])\hookrightarrow k[y_{T_i}]/(k[y_{T_i}]\cap [F])$$ of finitely generated $k$-algebras, corresponds to dominant morphism of affine $k$-schemes. Therefore, also the morphism $\pi_{T_i,T}\colon \Sol^\A_{T_i}(F)\to \Sol^\A_{T}(F)$ of affine $\kb$-varieties is dominant. By Chevalley's theorem (see e.g., \cite[Theorem 2.2.11]{Geck:AnIntroductionToAlgebraicGeometryAndAlgebraicGroups}) the image of a morphism of varieties is constructible. So the image of $\pi_{T_i,T}$ is a Zariski dense, constructible subset of $\Sol^\A_{T}(F)$. It therefore contains a subset $U$ that is open and Zariski dense in $\Sol^\A_{T}(F)$.
	Thus, every $a\in U$ extends to some $\widetilde{a}\in\Sol^\A_{T_i}(F)$. Via the embedding $\Sol^\A_{T_i}(F)\to \Sol^\P_i(F)$ we obtain an element $\widetilde{b}\in \Sol^\P_{i}(F)$ from $\widetilde{a}\in\Sol^\A_{T_i}(F)$. By Lemma \ref{lemma: projections surjective}, there exists a $b\in \Sol^\P(F)$ mapping to $\widetilde{b}\in \Sol^\P_{i}(F)$. This $b$ has the required property of Definition~\ref{defi: extends}.
\end{proof}

\subsection{Free sets and difference dimension}

% Let $K$ be an uncountable field containing $k$ as a subfield. (We do not require that $K$ is a difference field.)
We are now prepared to specify precisely how to count the degrees of freedom when determining sequence solutions.

\begin{prop} \label{prop: characterize free}
	Let $F\subseteq k\{y_1,\ldots,y_n\}$. For a finite subset $T$ of $\N\times\{1,\ldots,n\}$ the following conditions are equivalent:
	\begin{enumerate}
		\item There exists a Zariski dense open subset $U$ of $\A^T$ such that every $a\in U$ extends to a projective sequence solution of $F$.
		\item $\Sol^\A_T(F)=\A^T$.
		\item $k[y_T]\cap[F]=\{0\}$.	
		\item The image of $y_T$ in $k\{y_1,\ldots,y_n\}/[F]$ is algebraically independent over $k$.
	\end{enumerate}
\end{prop}
\begin{proof}
Let $U$ be as in (i). By Lemma \ref{lemma: extends implies solution} we have $U\subseteq \Sol^\A_T(F)\subseteq \A^T$. Since $U$ is Zariski dense in $\A^T$ and $\Sol^\A_T(F)$ is closed in $\A^T$, we see that $\Sol^\A_T(F)=\A^T$. So (i)$\Rightarrow$(ii). On the other hand, (ii)$\Rightarrow$(i) by Lemma \ref{lemma: U extends}.
Clearly, (iv) and (iii) are equivalent. Moreover, (iii)$\Leftrightarrow$(ii) by definition of $\Sol^\A_T(F)$.
\end{proof}

\begin{defi} \label{defi: free sets} Let $F\subseteq k\{y_1,\ldots,y_n\}$. A finite subset $T$ of $\N\times\{1,\ldots,n\}$ is \emph{free} with respect to $F$ if it satisfies the equivalent properties of Proposition \ref{prop: characterize free}.
\end{defi}

In Section \ref{subsec: A characterization of free sets in terms of affine} below we will obtain yet another characterization of free sets.
We next look at a couple of examples to familiarize ourselves with the definitions introduced above.

\begin{ex} \label{ex: infty continued}
	Let us return to Example \ref{ex: infinity}. So $F=\{y_1\s(y_1)-1\}$.  We have already seen that for $T=\{0\}$ every non-zero $a_0\in \mathbb{C}=\A^T$ extends to an affine sequence solution. Thus $T=\{0\}$ is free with respect to $F$.
 The element $a_0=0\in \A^T$ does not extend to an affine sequence solution but it extends to the projective sequence solution
	$$((1:0),(0:1),(1:0),(0:1),\ldots)\in(\P^1)^\N.$$ Indeed, for $i\geq 1$ and $T_i=\{0,\ldots,i\}$ we have
	$$\Sol^\A_{T_i}(F)=\{(a_0,a_0^{-1},\ldots,a_0^{\pm 1})|\ a_0\in \mathbb{C}\smallsetminus\{0\}\}\subseteq\A^{T_i}$$
	and $\Sol^\P_i(F)$ is obtained from $\Sol^\A_{T_i}(F)\simeq \A^1\smallsetminus\{0\}$ by adding two points, $((1:0),(0:1),\ldots)\in(\P^1)^{i+1}$ corresponding to the missing origin of $\A^1\smallsetminus\{0\}$ and $((0:1),(1:0),\ldots)\in(\P^1)^{i+1}$ corresponding to the missing point at infinity of $\A^1\smallsetminus\{0\}$. This shows that $\Sol^\P(F)\simeq \P^1$ is obtained from $\Sol^\A(F)\simeq \A^1\smallsetminus\{0\}$ by adding two points, namely $$((1:0),(0:1),\ldots) \text{ and } ((0:1),(1:0),\ldots)\in(\P^1)^\N.$$
	
	Note that $\Sol^\P(F)\subseteq (\P^1)^\N$ can also be described as the solution set of the multihomogeneous polynomials $$\s^i(y_1\s(y_1)-y_0\s(y_0))=\s^i(y_1)\s^{i+1}(y_1)-\s^i(y_0)\s^{i+1}(y_0), \ (i\geq 0).$$
	Every one-element subset $T$ of $\N$ is free with respect to $F$ but no subset of $\N$ with two or more elements is free. So, clearly, there is only one degree of freedom that should be counted in this example.
\end{ex}

\begin{ex} \label{ex: not constructible continued}
	Let us also revisit Example \ref{ex: not constructible}. So $F=\{\s(y_1)-y_1-1,\ y_1y_2-1\}$.
	For $T_i=\{0,\ldots,i\}\times \{1,2\}$ we have 
	$$\Sol^\A_{T_i}(F)=\left\{\begin{pmatrix}
		a & a+1 & \cdots & a+i \\
		a^{-1} & (a+1)^{-1} & \cdots & (a+i)^{-1}
	\end{pmatrix}\Big|\ a\in\mathbb{C}\smallsetminus\{0,-1,\ldots,-i\} \right\}\subseteq(\A^2)^{i+1}.
	$$
	So $\Sol^\A_{T_i}(F)\simeq \A^1\smallsetminus\{0,\ldots,-i\}$ and $\Sol^\P_i(F)\simeq\P^1$ is obtained from $\Sol_{T_i}^\A(F)$ by adding $i+2$ points at infinity. These are
	%$$\left((-n:(-n)^2:1),\ (-n+1:(-n+1)^2:1),\ldots,\ (-n+i:(-n+i)^2:1)\right)\in(\P^2)^{i+1}, $$
		$$\left((a:a^2:1),\ (a+1:(a+1)^2:1),\ldots,\ (a+i:(a+i)^2:1)\right)\in(\P^2)^{i+1}, $$
	where $a=0,\ldots,-i$, corresponding to the missing points $\{0,\ldots,-i\}$ and the point
	$$\left((0:1:0),(0:1:0),\ldots,(0:1:0)\right)\in(\P^2)^{i+1}, $$
	corresponding to the missing point at infinity.	
	 To explicitly describe $\Sol_i^\P(F)\subseteq (\P^2)^{i+1}$ set $X=\V^h(y_1y_2-y_0^2)=\V^h((y_1y_2-1)^h)\subseteq\P^2$. Note that $X$ is isomorphic to $\P^1$ (via $\P^1\to X,\ (a:b)\mapsto (ab:b^2:a^2)$) and that
		$$
		\xymatrix{
		X \ar^{\simeq}[rr] & & \P^1 \\
		& \A^1 \ar@{_(->}[lu] \ar@{^(->}[ru] &	
		}
		$$
	commutes, where $\A^1\hookrightarrow\P^1,\ a\mapsto (1:a)$ is the standard embedding and $\A^1\hookrightarrow X,\ a\mapsto (a:a^2:1)$ extends $\A^1\smallsetminus\{0\}\simeq \V(y_1y_2-1)\hookrightarrow X$. The automorphism $p\colon\A^1\to \A^1,\ a\mapsto a+1$ extends to an automorphism $p\colon X\to X$. We claim that 
	\begin{equation} \label{eq: psols for p}
	\Sol^\P_i(F)=\left\{(x,p(x),\ldots,p^i(x))\in(\P^2)^{i+1}|\ x\in X\right\}.
	\end{equation}
	The right-hand side of (\ref{eq: psols for p}) is closed in $(\P^2)^{i+1}$ and, by construction, it contains the image of $\Sol^\A_{T_i}(F)$ in $(\P^2)^{i+1}$. In fact, $\{(x,p(x),\ldots,p^i(x))\in(\P^2)^{i+1}|\ x\in X\}\simeq X\simeq \P^1$ is obtained from $\Sol^\A_{T_i}(F)\simeq \A^1\smallsetminus\{0,\ldots,-i\}$ by adding the $i+2$ points described above. This implies (\ref{eq: psols for p}).	
	Similarly, 
	$$\Sol^\A(F)=\left\{\begin{pmatrix}
	a & a+1 & \cdots &  \\
	a^{-1} & (a+1)^{-1} & \cdots &
	\end{pmatrix}\in(\A^2)^\N\Big|\ a\in\mathbb{C}\smallsetminus\{-n|\ n\in\N\} \right\}\subseteq(\A^2)^\N
	$$
	is in bijection with $\A^1\smallsetminus\{-n|\ n\in \N\}$ and 
	$$\Sol^\P(F)=\left\{(x,p(x),p^2(x),\ldots)\in(\P^2)^\N|\ x\in X\right\}\subseteq(\P^2)^\N$$
is in bijection with $\P^1$. So we obtain $\Sol^\P(F)$ from $\Sol^\A(F)$ by adding infinitely many points, namely, 
	$\left((0:1:0),(0:1:0),\ldots\right)\in(\P^2)^\N$ and
	$$\left((a:a^2:1),\ (a+1:(a+1)^2:1),\ldots,\ \right)\in(\P^2)^\N, $$
	where $a=0,-1,-2,\ldots$.
	
	 Note that, in general, for $F\subseteq k\{y_1,\ldots,y_n\}$ we have an inclusion $$ \Sol^\P(F)=\left\{a\in(\P^n)^\N\big|\ f(a)=0 \ \forall\ f \in [F]^h\right\}\subseteq \left\{a\in(\P^n)^\N\big|\ \s^i(f)(a)=0 \ \forall\ f \in F^h,\ i\in \N\right\}.$$
	However, this inclusion can be strict. Indeed, in the present example, the point 
	$$a=((0:1:0),(0:0:1),(0:0:1),\ldots)\in(\P^2)^\N$$ is a solution to $\s^i(y_1y_2-y_0^2)$ and $\s^i(\s(y_1)y_0-y_1\s(y_0)-y_0\s(y_0))$ for all $i\in \N$ but $a$ does not belong to $\Sol^\P(F)$ because $p(0:1:0)=(0:1:0)\neq (0:0:1)$. Of course, this can also be seen in terms of the equations: From $\s(y_1)\s(y_2)-1\in[F]$ and $\s(y_1)-y_1-1\in[F]$ we obtain $(y_1+1)\s(y_2)-1\in[F]$. Therefore $f=\s(y_2)y_1+\s(y_2)-1\in[F]$ but $f^h=\s(y_2)y_1+\s(y_2)y_0-y_0\s(y_0)$ does not vanish on $a$.

	For $T=\{(1,0)\}$, the set of all $a\in \A^T=\mathbb{C}$ that extend to an affine sequence solution of $F$ is $\mathbb{C}\smallsetminus\{-n|\ n\in\N\}$, which does not contain a non-empty Zariski open subset. The set of all $a\in \A^T=\mathbb{C}$ that extend to a projective sequence solution of $F$ is $\mathbb{C}\smallsetminus \{0\}$, which is Zariski open: Indeed, for $a\in\mathbb{C}\smallsetminus \{0\}$ the point $$b=b_a=\left((a:a^2:1),\ (a+1:(a+1)^2:1),\ldots,\ \right)\in(\P^2)^\N $$
	is a projective sequence solution of $F$ that extends $a$ because $(a:a^2:1)=(1:a:a^{-1})$. The point $a=0\in \A^T$ does not extend to a projective sequence solution of $F$ because the equation $y_1y_2-y_0^2=0$ does not have a solution with $y_0=1$ and $y_1=0$. This shows that in Lemma \ref{lemma: U extends} one cannot choose $U=\Sol^\A_T(F)$ in general.
	So $T=\{(0,1)\}$ is free with respect to $F$. More generally, every one-element subset of $\N\times \{1,2\}$ is free with respect to $F$ but no subset with two or more elements is free with respect to $F$. To see this, note that for a one-element subset $T$ of $\N\times \{1,2\}$, all elements of $\A^T\smallsetminus\{0\}$ extend to a projective sequence solution of $F$: For $T=\{(i,1)\}$, $b_{a-i}$ extends $a\in \A^T\smallsetminus\{0\}$ and for $T=\{(i,2)\}$, $b_{a^{-1}-i}$ extends $a\in \A^T\smallsetminus\{0\}$. No subset with two or more elements can be free because $\Sol^\P_i(F)\simeq \P^1$ is one dimensional for every $i\in \N$. Alternatively, for any two distinct elements in $\{y_1,y_2,\s(y_1),\s(y_2),\ldots\}$ we can always find a non-zero polynomial in $[\s(y_1)-y_1-1,\ y_1y_2-1]$ that only contains those two elements. So condition (iii) of Proposition \ref{prop: characterize free} is violated.
\end{ex}

\begin{ex} \label{ex: linear}
	Let $f=\s^m(y_1)+\lambda_{m-1}\s^{m-1}(y_1)+\ldots+\lambda_0y_1$ be a homogeneous linear difference polynomial over $k$. Then every $a=(a_0,\ldots,a_{m-1})\in\kb^m$ extends to an affine sequence solution via the recursive formula $a_{m+i}=\s^i(\lambda_{m-1})a_{m-1+i}+\ldots+\s^i(\lambda_0)a_i$ for $i\geq 0$. Thus $T=\{0,\ldots,m-1\}$ is free with respect to $F=\{f\}$. On the other hand, no subset of $\N$ containing more than $m$ elements is free with respect to $f$. So, overall, we count $m$ degrees of freedom.
	
	The same reasoning applies to any order $m$ difference polynomial of the form $f=\s^m(y_1)+g(y_1,\ldots,\s^m(y_1))$.
\end{ex}

\begin{ex} \label{ex: product}
	Let $f=y_1\s(y_1)$ over $(k,\s)=(\mathbb{C},\id)$. A sequence $a=(a_0,a_1,\ldots)\in\mathbb{C}^\N$ is an affine sequence solution if and only if $a_ia_{i+1}=0$ for $i\geq 0$, i.e., if every second entry is zero. For $m\geq 0$ the sets $T=\{0,2,4,\ldots,2m\}$ and $T=\{1,3,\ldots,2m+1\}$ are free with respect to $f$ but no subset of $\N$ containing two consecutive integers is free with respect to $f$.
\end{ex}

As in the above example, for a general system $F\subseteq k\{y_1,\ldots,y_n\}$ of algebraic difference equations one expects to encounter infinitely many degrees of freedom, when writing down a solution in the ring of sequences. Thus, to count them in a reasonable fashion, we need to count them asymptotically. For $i\geq 0$,
$$d_i(F)=\max\big\{|T| \ |\ T\subseteq \{0,\ldots,i\}\times\{1,\ldots,n\} \text{ is free w.r.t. } F\big\}$$
counts the degrees of freedom up to order $i$. To obtain a value between $0$ and $n$ we normalize $d_i(F)$ appropriately, i.e., we consider $0\leq\frac{d_i(F)}{i+1}\leq n$.

\begin{defi} \label{def: sdim for systems}
	Let $F\subseteq k\{y_1,\ldots,y_n\}$. In Corollary \ref{cor: limit exists} below it is shown that $$\sdim(F)=\lim_{i\to\infty}\frac{d_i(F)}{i+1}$$
	exists (inside $\mathbb{R}$). We call this limit the \emph{$\s$-dimension of $F$}.
\end{defi}
Note that by construction $\sdim(F)=\sdim([F])$ and $0\leq \sdim(F)\leq n$ for $F\subseteq k\{y_1,\ldots,y_n\}$. In Section \ref{sec: Comparison} we will compare $\sdim(F)$ with other notions of dimensions in difference algebra. In particular, we will show that our definition agrees with the standard definition via $\s$-transcendence bases whenever the latter notion applies. 

\begin{ex}
	For the sets $F$ in Examples \ref{ex: infty continued} and \ref{ex: not constructible continued} we have $d_i(F)=1$ for all $i\geq 0$ and so $\sdim(F)=0$. Also for $F$ as in Example \ref{ex: linear} $d_i(F)$ is bounded and so $\sdim(F)=0$. For $F=\{0\}\subseteq k\{y_1,\ldots,y_n\}$ one has $d_i(F)=n(i+1)$ and so $\sdim(F)=n$ as expected.
\end{ex}

The following example shows that $\sdim(F)$ does not need to be an integer.

\begin{ex}
 As in Example \ref{ex: product} let $F=\{y_1\s(y_1)\}$. For $i\geq 0$ even we have $d_i(F)=\frac{i}{2}$ and for $i$ odd we have $d_i(F)=\frac{i+1}{2}$. So $\sdim(F)=\lim_{i\to\infty}\frac{d_i(F)}{i+1}=\frac{1}{2}$.
\end{ex}
In Section \ref{sec:Covering Density and the dimension of difference monomials} we will determine the $\s$-dimension of a general univariate $\s$-monomial. Moreover, since the $\s$-dimension is not necessarily an integer it is natural to wonder which numbers occur. This question will be addressed in Section \ref{sec:Values of the difference dimension}.

\subsection{A characterization of free sets in terms of affine sequence solutions} \label{subsec: A characterization of free sets in terms of affine}

To complement Definition \ref{defi: free sets}, we deduce in this subsection a characterization of free sets that avoids projective sequence solutions. In fact, we show that, at least over an uncountable $\s$-field $k$, $T\subseteq \N\times \{1,\ldots,n\}$ is free with respect to $F\subseteq k\{y_1,\ldots,y_n\}$ if and only if the set of all $a\in\A^T$ that extend to an affine sequence solution of $F$ is Zariski dense in $\A^T$.

To also have a statement available for arbitrary $\s$-fields $k$, we fix an uncountable algebraically closed field $K$ containing $k$ as a subfield and we consider all solutions sets over $K$. For example, if $k$ is uncountable, we could choose $K=\kb$. Similarly to Subsection \ref{subsec: Affine sequence solutions}, we consider $K^\N$ as a \ks-algebra via $\s((a_i)_{i\in\N})=(a_{i+1})_{i\in\N}$ and $k\to K^\N,\ \lambda\mapsto (\s^i(\lambda))_{i\in\N}$. We set $\A^n_K=K^n$ and for $F\subseteq k\{y_1,\ldots,y_n\}$ we set
$$\Sol^{\A_K}(F)=\{a\in (K^\N)^n|\ f(a)=0 \ \forall \ f\in F\}=\V([F])\subseteq (\A^n_K)^\N.$$
For a finite subset $T$ of $\N\times\{1,\ldots,n\}$ we define
$$\Sol^{\A_K}_T(F)=\V([F]\cap k[y_T])\subseteq \A_K^T.$$

\begin{lemma} \label{lemma: Zariski dense}
The image of $\Sol^{\A_K}(F)$ in $\Sol^{\A_K}_T(F)$ is Zariski dense.
\end{lemma} 
\begin{proof}
	Let $g\in K[y_T]$ be a polynomial that vanishes on $\Sol^{\A_K}(F)$. We have to show that $g$ also vanishes on $\Sol^{\A_K}_T(F)$.
	
	There is a (strong) Nullstellensatz for polynomials in an arbitrary set of variables $Y$ (\cite{Lang:HilbertsNullstellensatzInInfiniteDimensionalSpace}). It states that for an algebraically closed field $K$ with $|K|>|Y|$, a polynomial $h\in K[Y]$ vanishes on all solutions of $H\subseteq K[Y]$ in $K^Y$ if and only if $h\in \sqrt{(H)}$. Therefore $g\in \sqrt{(F,\s(F),\ldots)}\subseteq K[\s^i(y_j)|\ (i,j)\in\N\times\{1,\ldots,n\}]$. Thus $g^m\in (F,\s(F),\ldots)=[F]\otimes_k K\subseteq k\{y_1,\ldots,y_n\}\otimes_kK$
 for some $m\geq 1$. Since $g\in K[y_T]=k[y_T]\otimes_k K$, it follows that
 $$g^m\in ([F]\otimes_k K)\cap (k[y_T]\otimes_k K)=([F]\cap k[y_T])\otimes_k K\subseteq k\{y_1,\ldots,y_n\}\otimes_k K.$$
 Thus $g^m$ vanishes on $\Sol^{\A_K}_T(F)$ and therefore also $g$ vanishes on $\Sol^{\A_K}_T(F)$.
\end{proof}

In \cite{Lang:HilbertsNullstellensatzInInfiniteDimensionalSpace} it is shown that the cardinality assumption $|K|>|Y|$ in the above proof is necessary for the Nullstellensatz in infinitely many variables. In fact, Lemma \ref{lemma: Zariski dense} does not hold without the assumption that $K$ is uncountable (Example \ref{ex: not Zariski dense}).

\begin{cor}
	Let $F\subseteq k\{y_1,\ldots,y_n\}$ and let $T$ be a finite subset of $\N\times\{1,\ldots,n\}$. Then $T$ is free with respect to $F$ if and only if the image of $\Sol^{\A_K}(F)$ in $\A_K^T$ is Zariski dense.
\end{cor}
\begin{proof}
	A polynomial in $[F]\cap k[y_T]$ vanishes on the image of $\Sol^{\A_K}(F)$ in $\A_K^T$. Thus, if the latter is Zariski dense in $\A_K^T$, then $[F]\cap k[y_T]=\{0\}$ and so $T$ is free with respect to $F$ (Proposition \ref{prop: characterize free}).
	
	On the other hand, if $T$ is free with respect to $F$, then $[F]\cap k[y_T]=\{0\}$ and so $\Sol_T^{\A_K}(F)=\A_K^T$. Thus the image of  $\Sol^{\A_K}(F)$ in $\A_K^T$ is Zariski dense by Lemma \ref{lemma: Zariski dense}.
\end{proof}

\section{The difference dimension of a difference algebra} \label{sec: The difference dimension}

In this section we introduce the $\s$-dimension $\sdim(R)$ of a finitely $\s$-generated \ks-algebra. We then show that, despite the fact that  $\sdim(R)$ need not be an integer, it satisfies many properties similar to the familiar case of finitely generated algebras over a field. For example, the difference dimension is compatible with tensor products and base change. For $F\subseteq k\{y_1,\ldots,y_n\}$ we have $\sdim(F)=\sdim(k\{y_1,\ldots,y_n\}/[F])$ and so results about the $\s$-dimension of $\s$-algebras have immediate corollaries for the $\s$-dimension of systems of algebraic difference equations.

\subsection{Recollection: Dimension of algebras}
	
Before defining the $\s$-dimension, we recall some well-known properties of the Krull dimension for finitely generated algebras over a field. (See, e.g., Sections 8 and 13 in \cite{Eisenbud:view}). This will be helpful for two reasons. Firstly, we will use these results in our later poofs and secondly, some of our results are difference analogs of these classical results about the Krull dimension.

Recall that the Krull dimension $\dim(R)$ of a ring $R$ is defined as the supremum over the lengths $n$ of all chains $\p_0\subsetneqq\p_1\subsetneqq\ldots\subsetneqq\p_n$ of prime ideals in $R$. For finitely generated algebras over a field, this supremum is finite and can be described through algebraically independent elements:

\begin{prop} \label{prop: dim and algebraic independence}
Let $R$ be an algebra over a field $k$ and let $A$ be a finite subset of $R$ such that $R=k[A]$.
Then 
\begin{equation} \label{eq: dim}
\dim(R)=\max\{|B|\ |\ B\subseteq A, \text{ $B$ is algebraically independent over $k$}\}.
\end{equation}

In particular, if $R$ is an integral domain, then $\dim(R)$ equals the transcendence degree of the field of fractions of $R$ over $k$. 
\end{prop}
\begin{proof}
	See \cite[Tag 00P0]{stacks-project} %\cite[Theorem A, Sec. 13.1]{Eisenbud:view} 
	for a proof that $\dim(R)$ equals the transcendence degree of the field of fractions of $R$ over $k$ in case $R$ is an integral domain. In general, let $d$ denote the value on the right hand side of equation (\ref{eq: dim}). From the definition of $\dim(R)$, it follows that $\dim(R)=\dim(R/\p)$ for some minimal prime ideal $\p$ of $R$. Since the image of $A$ in $R/\p$ generates the field of fractions of $R/\p$ as a field extensions of $k$, it contains a transcendence basis. So we may choose $B\subseteq A$ such that the image of $B$ is a transcendence basis of the field of fractions of $R/\p$ over $k$. Then $|B|=\dim(R/\p)=\dim(R)$. Since the image of $B$ in $R/\p$ is algebraically independent over $k$, also $B$ itself is algebraically independent over $k$. Therefore, $\dim(R)\leq d$.
	
	Conversely, assume that $B\subseteq A$ is algebraically independent over $k$ and $|B|=d$. Then $k[B]$ is a polynomial ring in $d$ variables. In particular, it is an integral domain. For any inclusion of rings $S_1\subseteq S_2$, any minimal prime ideal $\p_1$ of $S_1$ is of the form $\p_1=\p_2\cap S_1$ for some prime ideal $\p_2$ of $S_2$ (\cite[Chapter II, \S 2.6, Prop. 16]{Bourbaki:commutativealgebra}). Applying this to $k[B]\subseteq R$ with $\p_1$ the zero ideal of $k[B]$, we find a prime ideal $\p$ of $R$ with $\p\cap k[B]=\{0\}$. So $k[B]$ embeds into $R/\p$ and it follows that the transcendence degree of the field of fractions of $R/\p$ over $k$ is at least $|B|$. So, using \cite[Tag 00P0]{stacks-project} again, we obtain $d=|B|\leq\dim(R/\p)\leq\dim(R)$. Altogether, we obtain $\dim(R)=d$ as desired.
\end{proof}

The following lemma explains the behavior of Krull dimension under morphisms.

\begin{lemma} \label{lemma: dim under morphisms}
	Let $R$ and $S$ be finitely generated $k$-algebras.
	\begin{enumerate}
		\item If there exists an injective morphism $R\to S$ of $k$-algebras, then $\dim(R)\leq\dim(S)$.
		\item If there exists a surjective morphism $R\to S$ of rings, then $\dim(R)\geq\dim(S)$. 
	\end{enumerate}
\end{lemma}
\begin{proof}
	For (i), note that a finite generating set $A$ of $R$ can be extended to a finite generating set of $S$. An algebraically independent subset of $A$ remains algebraically independent in $S$ by the injectivity of $R\to S$. Thus the claim follows from Proposition \ref{prop: dim and algebraic independence}.
	
	Claim (ii) follows from the fact that prime ideals in $S$ are in bijection with prime ideals in $R$ containing the kernel of $R\to S$.
\end{proof}

The Krull dimension is additive with respect to the tensor product:

\begin{lemma} \label{lemma: dim and tensor}
	Let $R$ and $S$ be finitely generated $k$-algebras. Then $\dim(R\otimes_k S)=\dim(R)+\dim(S)$.
\end{lemma}
\begin{proof}
	Let $A\subseteq R$ and $B\subseteq S$ be finite such that $R=k[A]$ and $S=k[B]$. Set 
	$C=\{a\otimes 1|\ a\in A\}\cup \{1\otimes b|\ b\in B\}$.
	%$A\otimes1=\{a\otimes 1|\ a\in A\}\subseteq  R\otimes_k B$ and similarly for $1\otimes B$. 
	Then $k[C]=R\otimes_k S$ and a subset $C'$ of $C$ is algebraically independent over $k$ if and only if
	$A'=\{a\in A|\ a\otimes 1\in C'\}$ and $B'=\{b\in B|\ 1\otimes b\in C'\}$ are algebraically independent over $k$. Therefore $|C'|$ is maximal if and only if $|A'|$ and $|B'|$ is maximal. So the claim follows from Proposition \ref{prop: dim and algebraic independence}.
\end{proof}

The Krull dimension is invariant under base change:

\begin{lemma}[{\cite[Tag 00P3]{stacks-project}}] \label{lemma: dim and base change}
	Let $k'/k$ be a field extension and $R$ a finitely generated $k$-algebra. Then $\dim(R\otimes_k k')=\dim(R)$.
\end{lemma}

Taking the quotient by the nilradical does not affect the Krull dimension:

\begin{lemma} \label{lemma: dim and nilradical}
Let $R$ be a finitely generated $k$-algebra and $R_{\operatorname{red}}=R/\sqrt{0}$ the quotient of $R$ by the nilradical $\sqrt{0}$ of $R$. Then $\dim(R_{\operatorname{red}})=\dim(R)$.
\end{lemma}
\begin{proof}
	The nilradical $\sqrt{0}$ is contained in every prime ideal of $R$.
\end{proof}

\subsection{Difference dimension of difference algebras}

We first show that the limit from Definition \ref{def: sdim for systems} exists. To achieve this we will use the following well-known elementary lemma. See, e.g., \cite[Prop. 10.7]{DenkerGrillenbergerSigmund:ErgodicTheoryOnCompactSpaces}.

%??
%
%Throughout $k$ denotes a $\s$-field. For a finitely generated $k$-algebra $S$, $\dim(S)$ denotes its Krull dimension. 
%
%For an ideal $I$ of $S$ we also write $\dim(I)$ for the dimension of the corresponding affine variety, i.e. $\dim(I)=\dim(S/I)$.
%
%
% tensor product (base extension), inversive, $\N$ contains 0$, $\trdeg$, $k(\p)$,  $\s$-domain

\begin{lemma}[Fekete's Subadditive Lemma] \label{lemma:subadditive}
	If $(e_i)_{i\geq 1}$ is a sequence of non-negative real numbers that is subadditve, i.e., $e_{i+j}\leq e_i+e_j$ for all $i,j\geq 1$, then $\lim_{i\to\infty} \frac{e_i}{i}$ exists (inside $\mathbb{R}$) and is equal to $\inf \frac{e_i}{i}$.
\end{lemma}

%
%We will need a simple lemma:
%\begin{lemma}\label{lemma:bounddim}
%Let $R\subset S$ be finitely generated $k$-algebras such that $S$ can be generated over $R$ by $n$ elements. Then
%$$\dim(S)\leq\dim(R)+n.$$
%\end{lemma}
%\begin{proof}
%Let $\p\subset S$ be a prime ideal with $\trdeg(k(\p)|k)=\dim(S)$ and set $\q=\p\cap R$. Then
%$$\dim(S)=\trdeg(k(\p)|k)=\trdeg(k(\p)|k(\q))+\trdeg(k(\q)|k)\leq n+\dim(R).$$
%\end{proof}

The following theorem allows us to define a meaningful notion of $\s$-dimension for \emph{any} finitely $\s$-generated \ks-algebra.

\begin{theo} \label{theo:sdimdef}
	Let $R$ be a finitely $\s$-generated $k$-$\s$-algebra. Choose a finite subset $A$ of $R$ such that $R=k\{A\}$ and set
	$d_i=\dim(k[A,\ldots,\s^{i}(A)])$ for $i\geq 0$. Then the limit
	$$d=\lim_{i\to\infty}\frac{d_i}{i+1}$$ exists (inside $\R$) and does not depend on the choice of $A$.
\end{theo}
\begin{proof}
	As the first step, we will show that we can assume without loss of generality that $k$ is inversive. Let $k^*$ denote the inversive closure of $k$ (\cite[Def. 2.1.6]{Levin}) and set $R'=R\otimes_k k^*$. Then $A'=\{a\otimes1 |\ a\in A\}$ $\s$-generates $R'$ over $k^*$. Set
	$d_i'=\dim(k^*[A',\ldots,\s^{i}(A')])$ for $i\geq 0$. As $k^*[A',\ldots,\s^{i}(A')]=k[A,\ldots,\s^{i}(A)]\otimes_k k^*$ we have $d_i=d'_i$ for $i\geq0$.
	So, we can assume that $k$ is inversive.

	To show that $\lim_{i\to\infty}\frac{d_i}{i+1}$ exists, it suffices to show that the sequence $(e_i)_{i\in\N}=(d_{i-1})_{i\in\N}$ is subadditive, because then $$\lim_{i\to\infty}\frac{d_i}{i+1}=\lim_{i\to\infty}\frac{d_{i-1}}{i}=\lim_{i\to\infty}\frac{e_i}{i}$$ exists by Lemma~\ref{lemma:subadditive}. Let $i,j\geq 1$. Since $k$ is inversive, the map $$\s^i\colon k[A,\ldots,\s^{j-1}(A)]\to k[\s^i(A),\ldots,\s^{i+j-1}(A)]$$ is surjective. Thus $\dim(k[\s^i(A),\ldots,\s^{i+j-1}(A)])\leq d_{j-1}=e_j$ by Lemma \ref{lemma: dim under morphisms} (ii). The canonical map
	$$k[A,\ldots,\s^{i-1}(A)]\otimes_k k[\s^i(A),\ldots,\s^{i+j-1}(A)]\longrightarrow k[A,\ldots,\s^{i+j-1}(A)]$$ is also surjective. Therefore, using Lemma \ref{lemma: dim under morphisms} (ii) and Lemma \ref{lemma: dim and tensor}, we find

	 $$e_{i+j}\leq e_i+\dim(k[\s^i(A),\ldots,\s^{i+j-1}(A)])\leq e_i+e_j.$$

	It remains to show that $d=\lim_{i\to\infty}\frac{d_i}{i+1}$ does not depend on the choice of the $\s$-generating set $A$. This is similar to \cite[Prop. A.24]{DiVizioHardouinWibmer:DifferenceGaloisofDifferential} but we include the argument for the sake of completeness. 
	So let $A'\subseteq R$ be another finite set such that $R=k\{A'\}$ and set
	$d_i'=\dim(k[A',\ldots,\s^{i}(A')])$ for $i\geq 0$. Then $A'\subseteq k[A,\ldots,\s^{j}(A)]$ for some $j\geq 0$ and therefore $k[A',\ldots, \s^{i}(A')]\subseteq k[A,\ldots,\s^{i+j}(A)]$. Thus $d_i'\leq d_{i+j}$ by Lemma \ref{lemma: dim under morphisms}(i).
	
 If $B$ is an algebraically independent subset of $A\cup\ldots\cup \s^{i+j}(A)$ such that $|B|=d_{i+j}$, then $B\cap (A\cup\ldots\cup\s^i(A))$ is an algebraically independent subset of $A\cup\ldots\cup\s^i(A)$ and therefore $|B\cap (A\cup\ldots\cup\s^i(A))|\leq d_i$ by Proposition \ref{prop: dim and algebraic independence}. Thus 
		$$d_{i+j}=|B|\leq |B\cap (A\cup\ldots\cup\s^i(A))|+|B\cap (\s^{i+1}(A)\cup\ldots\cup\s^{i+j}(A))|\leq d_i+|A|j.$$	
%	As $k[A,\ldots,\s^{i+j}(A)]$ is generated by $|A|j$ elements over $k[A,\ldots,\s^{i}(A)]$, we see, using Proposition \ref{prop: dim and algebraic independence}), that $d_{i+j}\leq d_i+|A|j$. 
	So
	$$\frac{d'_i}{i+1}\leq\frac{d_{i+j}}{i+1}\leq\frac{d_i}{i+1}+\frac{|A|j}{i+1}.$$
	Since $\lim_{i\to\infty}\frac{|A|j}{i+1}=0$, it follows that $\lim_{i\to\infty}\frac{d'_i}{i+1}\leq \lim_{i\to\infty}\frac{d_i}{i+1}$.
\end{proof}

\begin{defi} \label{defi: sdim}
Let $R$ be a finitely $\s$-generated $k$-$\s$-algebra. The real number $d\geq 0$ defined in Theorem \ref{theo:sdimdef} above is called the \emph{$\s$-dimension} of $R$. We denote it by $\sdim(R).$
% If $F\subseteq k\{y_1,\ldots,y_n\}$ is a system of algebraic difference equations we set $\sdim(F)=\sdim(k\{y_1,\ldots,y_n\}/[F])$.
\end{defi}

We note that the idea to consider the sequence $\frac{d_i}{i+1}$ already appears in \cite[A~7]{DiVizioHardouinWibmer:DifferenceGaloisofDifferential}. There, the $\s$-dimension is defined as  $\lfloor\limsup_{i\to\infty} \frac{d_i}{i+1}\rfloor$ and it is shown (\cite[Prop.~A.24]{DiVizioHardouinWibmer:DifferenceGaloisofDifferential}) that $\limsup_{i\to\infty} \frac{d_i}{i+1}$ does not depend on the choice of the finite $\s$-generating set. Here $\lfloor x\rfloor$ is the floor of $x$, i.e., the largest integer not greater than $x$. Theorem \ref{theo:sdimdef} shows that there is no need to consider the limes superior since indeed the limit exists.

The floor of the limes superior was taken in \cite{DiVizioHardouinWibmer:DifferenceGaloisofDifferential} simply to obtain an integer value. The dimension of an algebraic variety is always an integer and so it may seem natural to also only allow integer values for the dimension in difference algebraic geometry. However, omitting the floor function makes the invariant stronger: Two difference algebras with distinct difference dimensions cannot be isomorphic and not using the floor allows us to recognize more difference algebras as non-isomorphic. 

Moreover, while non-integer values for the dimension may look unusual to the algebraist, in discrete dynamics, it is very common to consider numerical invariants that are not necessarily integers, for example, the topological entropy and the mean dimension need not be integers. In fact, our notion of difference dimension can be seen as an algebraic version of mean dimension. Mean dimension was first introduced by M. Gromov in \cite{Gromov:TopologicalInvariantsOfDynamicalSystemsAndSpacesOfHolomorphicMapsI}
and curiously enough, in Section 0.7 he writes: ``The present notion of mean dimension(s) arose from my attempts to geometrize the algebraic and model theoretic conceptions of dimensions over difference fields.''
We note that \cite{Gromov:TopologicalInvariantsOfDynamicalSystemsAndSpacesOfHolomorphicMapsI} is mainly concerned with compact metric spaces but as pointed out in Section 1.9.3 and remark \emph{On extension of Prodim to Nontoplogical Categories} right before Section 1.9.7 in \cite{Gromov:TopologicalInvariantsOfDynamicalSystemsAndSpacesOfHolomorphicMapsI}, some definitions and constructions there, also make sense in some algebraic categories. Our definition of difference dimension is more or less the same as the definition of projective dimension in \cite[Section 1.9]{Gromov:TopologicalInvariantsOfDynamicalSystemsAndSpacesOfHolomorphicMapsI}, a quantity closely related to the mean dimension.
%See the beginning of Section 1.9 and the remark \emph{On extension of Prodim to Nontoplogical Categories} right before Section 1.9.7 there. 
To make the connection between the two definitions, note that in \cite{Gromov:TopologicalInvariantsOfDynamicalSystemsAndSpacesOfHolomorphicMapsI} the base difference field $k$ is assumed to be constant, i.e., $\s\colon k\to k$ is the identity map. To match the notation in the beginning of \cite[Section~1.9]{Gromov:TopologicalInvariantsOfDynamicalSystemsAndSpacesOfHolomorphicMapsI} replace the group $\Gamma$ there with the monoid $\N$ and set $\Omega_i=\{0,\ldots,i\}$ for $i\in \N$. Moreover, choose $\underline{X}=\A^n$ so that $X=\underline{X}^\Gamma=(\A^n)^\N$. For $F\subseteq k\{y_1,\ldots,y_n\}$ (as in Section \ref{sec: counting}) set $Y=\Sol^\A(F)\subseteq X$ and $Y|\Omega_i=\Sol^\A_{T_i}(F)$, where $T_i=\{0,\ldots,i\}\times \{1,\ldots,n\}$. Then
$$\operatorname{prodim}(Y : \{\Omega_i\})=\liminf_{i\to \infty}\dim(Y|\Omega_i)/|\Omega_i|$$
from \cite{Gromov:TopologicalInvariantsOfDynamicalSystemsAndSpacesOfHolomorphicMapsI} becomes the limit in our Definition \ref{def: sdim for systems}.

In Section \ref{sec: Comparison} below we will compare Definition \ref{defi: sdim} with other notions of dimension in difference algebra. In particular, we will show (Proposition \ref{prop: compare stredeg}) that $\sdim(R)$ agrees with the $\s$-transcendence degree over $k$ of the field of fractions of $R$ in case $R$ is an integral domain with $\s\colon R\to R$ injective.

We can now justify Definition \ref{def: sdim for systems}. 

\begin{cor} \label{cor: limit exists}
	Let $F\subseteq k\{y_1,\ldots,y_n\}$ and for $i\geq 0$ set
	$$d_i(F)=\max\big\{|T| \ |\ T\subseteq \{0,\ldots,i\}\times\{1,\ldots,n\} \text{ is free w.r.t. } F\big\}.$$
	Then $d=\lim_{i\to\infty}\frac{d_i(F)}{i+1}$
	exists.
\end{cor}
\begin{proof}
	Set $R=k\{y_1,\ldots,y_n\}/[F]$ and let $A=\{a_1,\ldots,a_n\}$ denote the image of $\{y_1,\ldots,y_n\}$ in $R$. Recall (Proposition \ref{prop: characterize free}) that $T\subseteq \N\times\{1,\ldots,n\}$ is free with respect to $F$ if and only if $\{\s^i(a_j)|\ (i,j)\in T\}$ is algebraically independent over $k$.
	
%		For a $k$-algebra $S$, generated by a finite set $A$, one has $$\dim(S)=\max\{|B|\ |\ B\subseteq A, \text{ $B$ is algebraically independent over $k$}\}.$$
	Therefore, Proposition \ref{prop: characterize free} implies $d_i(F)=\dim(k[A,\ldots,\s^i(A)])$ for $i\geq 0$ and the claim follows from Theorem~\ref{theo:sdimdef}.
\end{proof}

 Note that in Theorem \ref{theo:sdimdef} and Corollary \ref{cor: limit exists} the limit of the sequence is in fact the infimum of the sequence. This follows from Lemma \ref{lemma:subadditive} and the proofs of Theorem \ref{theo:sdimdef} and Corollary \ref{cor: limit exists}.
From the proof of Corollary \ref{cor: limit exists} we also obtain:

\begin{rem} \label{rem: cut at i}
	For $i\geq 0$ set $k\{y\}[i]=k[y_1,\ldots,y_n,\ldots,\s^i(y_1),\ldots,\s^i(y_n)]$ and for a $\s$-ideal $I$ of $k\{y_1,\ldots,y_n\}$ set $I[i]=I\cap k\{y\}[i]$. We have
	$$\sdim(I)=\sdim(k\{y_1,\ldots,y_n\}/I)=\lim_{i\to\infty}\frac{d_i}{i+1},$$
	where $d_i=\dim(k\{y\}[i]/I[i])$.
\end{rem}

\begin{ex}
	Let $R$ be a \ks-algebra that is finitely generated as a $k$-algebra. Then $\sdim(R)=0$. To see this, note that if $A$ generates $R$ as a $k$-algebra, then also $k\{A\}=R$ and so $d_i=\dim(R)$ for $i\geq 0$.
\end{ex}

The following proposition shows that our notion of $\s$-dimension generalizes the usual notion of dimension in algebraic geometry.
\begin{prop} \label{prop: sdim for algebraic}
	Let $F\subseteq k[y_1,\ldots,y_n]\subseteq k\{y_1,\ldots,y_n\}$ be a system of algebraic equations. Then $\sdim(F)$ equals the dimension of the algebraic variety defined by $F$.
\end{prop}
\begin{proof}
	Let $X$ be the algebraic variety defined by $F$ and $d=\dim(X)$. For $i\geq 0$, the algebraic variety defined by $\s^i(F)\subseteq k[\s^i(y_1),\ldots,\s^i(y_n)]$ is the base change of $X$ via $\s^i\colon k\to k$. In particular, it also has dimension $d$ (cf. Lemma \ref{lemma: dim and base change}). So
	$$(F,\s(F),\ldots,\s^i(F))\subseteq k[y_1,\ldots,y_n,\ldots,\s^i(y_1),\ldots,\s^i(y_n)]$$
	defines an $(i+1)$-fold product of varieties of dimension $d$, i.e., a variety of dimension $d(i+1)$ (cf. Lemma \ref{lemma: dim and tensor}).
	
	We next show that, with the notation of Remark \ref{rem: cut at i}, we have
	\begin{equation} \label{eq: order i}
	[F][i]=(F,\s(F),\ldots,\s^i(F))\subseteq  k\{y\}[i]
	\end{equation}
	 for all $i\in \N$. Clearly, $(F,\s(F),\ldots,\s^i(F))\subseteq [F][i]$. So let us establish the reverse inclusion. To this end, note that for a $k$-algebra $S$, a set of indeterminates $Y$ over $S$ and an ideal $I$ of $k[Y]$ one has $(I)\cap k[Y]=I$, where $(I)\subseteq S[Y]$ denotes the ideal of $S[Y]$ generated by $I$. (This follows from $S[Y]=S\otimes_k k[Y]$ and the fact that the tensor product has this property. See, e.g., \cite[Lemma 1.4.5 ]{DuascualescuNuastuasescu:HopfAlgebras}). We will apply this with  $S=k[\s^{i+1}(y_1),\ldots,\s^{i+1}(y_n),\s^{i+2}(y_1),\ldots]/(\s^{i+1}(F),\s^{i+2}(F),\ldots)$,
	  $Y=\{y_1,\ldots,y_n,\ldots,\s^i(y_1),\ldots,\s^i(y_n)\}$ and $I=(F,\ldots,\s^i(F))\subseteq k[Y]=k\{y\}[i]$. The image of any $h\in [F]$ in $S[Y]$ lies in $(I)\subseteq S[Y]$, because an element in $\s^j(F)$ ($j\geq i+1$) becomes zero in $S$. If, moreover, $h\in [F][i]$, then $h\in k[Y]$, and so $h\in (I)\cap k[Y]=I$. This proves (\ref{eq: order i}).

	Thus, if $A$ denotes the image of $\{y_1,\ldots,y_n\}$ in $k\{y_1,\ldots,y_n\}/[F]$, then
	$$k[A,\ldots,\s^i(A)]=k\{y\}[i]/[F][i]=k\{y\}[i]/(F,\ldots,\s^i(F))$$
	has dimension $d_i=d(i+1)$. Therefore $$\sdim(F)=\sdim(k\{A\})=\lim_{i\to\infty}\frac{d_i}{i+1}=d.$$	
\end{proof}

We will next establish some elementary properties of the $\s$-dimension which show that it behaves as one may expect from a notion of dimension. Most of these properties follow rather directly from the corresponding property of finitely generated algebras.

\begin{prop} \label{prop: morphism and sdim}
	Let $R$ and $S$ be finitely $\s$-generated $k$-$\s$-algebras.
	\begin{enumerate}
		\item If there exists an injective morphism $R\to S$ of $k$-$\s$-algebras, then $\sdim(R)\leq\sdim(S)$.
		\item If there exists a surjective morphism $R\to S$ of $k$-$\s$-algebras, then $\sdim(R)\geq\sdim(S)$.
	\end{enumerate}
\end{prop}
\begin{proof}
	(i): We may assume that $R$ is a \ks-subalgebra of $S$. Let $A$ be a finite $\s$-generating set for $R$. Then we can extend $A$ to a finite $\s$-generating set $B$ of $S$. For $i\geq 0$ we have $k[A,\ldots,\s^{i}(A)]\subseteq k[B,\ldots,\s^{i}(B)]$ and therefore,  using Lemma \ref{lemma: dim under morphisms} (i),

 $$\dim(k[A,\ldots,\s^{i}(A)])\leq\dim(k[B,\ldots,\s^{i}(B)]).$$
	Thus $\sdim(R)\leq\sdim(S)$.
	
	(ii): Let $A\subseteq R$ be finite such that $R=k\{A\}$ and let $\overline{A}$ denote the image of $A$ in $S$ under a surjective morphism. Then $k\{\overline{A}\}=S$. Since $k[A,\ldots,\s^{i}(A)]$ surjects onto $k[\overline{A},\ldots,\s^{i}(\overline{A})]$ for $i\geq 0$, we see, using Lemma \ref{lemma: dim under morphisms} (ii), that $$\dim(k[A,\ldots,\s^{i}(A)])\geq\dim(k[\overline{A},\ldots,\s^{i}(\overline{A})]),$$
	and therefore $\sdim(R)\geq \sdim(S)$.
\end{proof}

In terms of systems of algebraic difference equations Proposition \ref{prop: morphism and sdim} has the following interpretation:
\begin{cor} \label{cor: morphism and sdim}
	\begin{enumerate}\item If $F\subseteq k\{y_1,\ldots,y_n\}$ and $G\subseteq k\{y_1,\ldots,y_n,z_1,\ldots,z_m\}$ are such that $[G]\cap k\{y_1,\ldots,y_n\}=[F]$, then $\sdim(F)\leq\sdim(G)$.
	\item If $F,G\subseteq k\{y_1,\ldots,y_n\}$ are such that $[F]\subseteq [G]$ (e.g., $F\subseteq G$), then $\sdim(F)\geq\sdim(G)$.
\end{enumerate}
	 \qed
\end{cor}

Like the Krull dimension of finitely generated algebras our $\s$-dimension is additive with respect to the tensor product.

\begin{prop} \label{prop: sdim and tensor products}
	Let $R$ and $S$ be finitely $\s$-generated $k$-$\s$-algebras. Then
	$$\sdim(R\otimes_k S)=\sdim(R)+\sdim(S).$$
\end{prop}
\begin{proof}
Let $A$ and $B$ be finite $\s$-generating sets for $R$ and $S$ respectively. Then $C=\{a\otimes 1|\ a\in A\}\cup\{1\otimes b|\ b\in B\}$ is a finite $\s$-generating set for $R\otimes_k S$. Moreover, for $i\geq 0$ we have $k[C,\ldots,\s^{i}(C)]=k[A,\ldots,\s^{i}(A)]\otimes_k k[B,\ldots,\s^{i}(B)]$ and therefore, using Lemma~\ref{lemma: dim and tensor}, $$\dim(k[C,\ldots,\s^{i}(C)])=\dim(k[A,\ldots,\s^{i}(A)])+\dim(k[B,\ldots,\s^{i}(B)]).$$
\end{proof}
In terms of systems of algebraic difference equations, Proposition \ref{prop: sdim and tensor products} has the following interpretation:
\begin{cor}
 If $F\subseteq k\{y_1,\ldots,y_n\}$ and $G\subseteq k\{z_1,\ldots,z_m\}$, then $F\cup G\subseteq k\{y_1,\ldots,y_n,z_1,\ldots,z_m\}$ has $\s$-dimension $\sdim(F)+\sdim(G)$. \qed
\end{cor}

The following proposition shows that our notion of $\s$-dimension is well-behaved under extension of the base $\s$-field (cf. \cite[Lemma A.27]{DiVizioHardouinWibmer:DifferenceGaloisofDifferential}).
\begin{prop} \label{prop: sdim and base extension}
	Let $R$ be a finitely $\s$-generated $k$-$\s$-algebra. Let $k'$ be a $\s$-field extension of $k$ and consider $R'=R\otimes_k k'$ as a $k'$-$\s$-algebra. Then
	$$\sdim(R')=\sdim(R).$$
\end{prop}
\begin{proof}
	If $A\subseteq R$ is a finite $\s$-generating set for the \ks-algebra $R$, then $A'=\{a\otimes 1|\ a\in A\}$ is a finite $\s$-generating set for the $k'$-$\s$-algebra $R'$. Moreover, $\dim(k[A,\ldots,\s^{i}(A)])=\dim(k'[A',\ldots,\s^{i}(A')])$ for $i\geq 0$ by Lemma~\ref{lemma: dim and base change} since $k'[A',\ldots,\s^{i}(A')]=k[A,\ldots,\s^{i}(A)]\otimes_k k'$ .
\end{proof}

In terms of systems of algebraic difference equations, Proposition \ref{prop: sdim and base extension} has the following interpretation:
\begin{cor}
 Let $k'$ be a $\s$-field extension $k$ and $F\subseteq k\{y_1,\ldots,y_n\}$. Then the $\s$-dimension of $F$ considered as a subset of $k\{y_1,\ldots,y_n\}$ agrees with the $\s$-dimensions of $F$ considered as a subset of $k'\{y_1,\ldots,y_n\}$.\qed
\end{cor} 

For a $\s$-ring $R$, the nilradical $\sqrt{0}\subseteq R$ of $R$ is a $\s$-ideal. Therefore $R_{\operatorname{red}}:=R/\sqrt{0}$ has naturally the structure of a $\s$-ring. As in commutative algebra, passing from $R$ to $R_{\operatorname{red}}$ does not affect the dimension:
\begin{prop} \label{prop: sdim and radical}
	Let $R$ be a finitely $\s$-generated $k$-$\s$-algebra. Then
	$$\sdim(R_{\operatorname{red}})=\sdim(R).$$
\end{prop}
\begin{proof}
	Let $A\subseteq R$ be a finite $\s$-generating set for $R$ and let $\overline{A}$ denote the image of $A$ in $R_{\operatorname{red}}$. Then $\overline{A}$ is a finite $\s$-generating set for  $R_{\operatorname{red}}$ and
	$k[\overline{A},\ldots,\s^{i}(\overline{A})]=k[A,\ldots,\s^{i}(A)]_{\operatorname{red}}$ for $i\geq 0$. Therefore $\dim(k[\overline{A},\ldots,\s^{i}(\overline{A})])=\dim(k[A,\ldots,\s^{i}(A)])$  by Lemma \ref{lemma: dim and nilradical}.
\end{proof}	

In terms of systems of algebraic difference equations Proposition \ref{prop: sdim and radical} can be reinterpreted as:
\begin{cor} \label{cor: radical}
	Let $F\subseteq k\{y_1,\ldots,y_n\}$. Then
	$$\sdim(F)=\sdim([F])=\sdim(\sqrt{[F]}).$$	\qed
\end{cor}

%
%\begin{lemma}??
%	Let $R$ be a finitely $\s$-generated $k$-$\s$-algebra. Then
%	$$\sdim(R/\ker(\s))=\sdim(R).$$
%\end{lemma}
%

\section{Comparison with other notions of dimension} \label{sec: Comparison}

In this section we compare our notion of $\s$-dimension with two other notions in the literature. Firstly, we show that our notion generalizes the standard definition via $\s$-transcendence bases. Secondly, we show that our $\s$-dimension is an upper bound for the difference Krull dimension.

Let us first recall some basic facts about the $\s$-transcendence degree (\cite[Section~4.1]{Levin}). Let $R$ be a \ks-algebra. A subset $A$ of $R$ is \emph{$\s$-algebraically independent} (over $k$) if the family $(\s^i(a))_{a\in A, i\in\N}$ is algebraically independent over $k$. If $K$ is a $\s$-field extension of $k$, a maximal $\s$-algebraically independent subset is called a \emph{$\s$-transcendence basis} of $K/k$. Any two $\s$-transcendence bases have the same cardinality, which is called the \emph{$\s$-transcendence degree} of $K/k$.

Also recall that a $\s$-ideal $I$ of a $\s$-ring $R$ is \emph{reflexive} if $\s^{-1}(I)=I$. (This implies that $\s\colon R/I\to R/I$ is injective.)
In \cite[Definition 4.2.21]{Levin} the difference dimension of a prime reflexive $\s$-ideal $I$ of $k\{y_1,\ldots,y_n\}$ is defined as the $\s$-transcendence degree of the fraction field of $k\{y_1,\ldots,y_n\}/I$ over $k$. (We will see in a moment that our $\sdim(I)$ agrees with this definition, so there is no ambiguity with the naming.)

The following proposition shows that our definition of $\s$-dimension agrees with the classical definition whenever the latter applies, i.e., when $R$ is an integral domain with $\s\colon R\to R$ injective (cf. \cite[Lemma A.26]{DiVizioHardouinWibmer:DifferenceGaloisofDifferential}).
\begin{prop} \label{prop: compare stredeg}
	Let $R$ be a finitely $\s$-generated $k$-$\s$-algebra. Assume that $R$ is an integral domain. Then $\sdim(R)$ equals the largest integer $n$ such that there exist $n$ $\s$-algebraically independent elements inside $R$. Moreover, if $\s\colon R\to R$ is injective, $\sdim(R)$ equals the $\s$-transcendence degree of the field of fractions of $R$ over $k$.
\end{prop}
\begin{proof}
	Let $A$ be a finite subset of $R$ such that $R=k\{A\}$ and set $d_i=\dim(k[A,\ldots,\s^i(A)])$ for $i\geq 0$. In \cite[Lemma and Definition 4.21]{Hrushovski:elementarytheoryoffrobenius} (cf. \cite[Theorem 5.1.1]{Wibmer:AlgebraicDifferenceEquations}) it is shown that there exist $d,e\in\N$ such that $d_i=d(i+1)+e$ for $i\gg 0$.
	Moreover, $d$ is the $\s$-transcendence degree over $k$ of the field of fractions $K$ of $R/(0)^*$, where $$(0)^*=\{r\in R|\ \exists\ m\geq 1 : \s^m(r)=0\}.$$
	Note that because $R$ is an integral domain, $(0)^*$ is a (reflexive) prime ideal and $K$ is a $\s$-field extension of $k$. We have
	$$\sdim(R)=\lim_{i\to\infty}\frac{d_i}{i+1}=\lim_{i\to\infty}\frac{d(i+1)+e}{i+1}=d.$$
	
	If $a_1,\ldots,a_n\in R$ are $\s$-algebraically independent over $k$, then $k\{a_1,\ldots,a_n\}\cap(0)^*=\{0\}$, because $\s$ is injective on  $k\{a_1,\ldots,a_n\}$. Thus $k\{a_1,\ldots,a_n\}$ embeds into $K$ and it follows that $n\leq d$. 
	
	On the other hand, we can choose a $\s$-transcendence basis $b_1,\ldots,b_d$ of $K/k$ that is contained in $R/(0)^*$. If $a_1,\ldots,a_d\in R$ are such that they are mapped onto $b_1,\ldots,b_d$, then $a_1,\ldots,a_d\in R$ are $\s$-algebraically independent over $k$. It follows that $d=\sdim(R)$ is the largest integer such that there exist $d$ $\s$-algebraically independent elements in $R$. 
	
	If $\s\colon R\to R$ is injective, then $(0)^*=\{0\}$ and $K$ equals the field of fractions of $R$.	
\end{proof}

%The above proposition shows that $\sdim(R)$ is an integer if $R$ is an integral domain. In general $\sdim(R)$ need not be an integer. For example, it follows from Example \ref{ex:dimrational} that $\sdim(k\{y\}/[y\s(y)])=\frac{1}{2}$. More generally we have:

Recall that a $\s$-ideal $I$ of a $\s$-ring $R$ is \emph{perfect} if $f\s(f)\in I$ implies $f\in I$ for all $f\in R$. Perfect $\s$-ideals are important in classical difference algebra because they feature prominently in a difference Nullstellensatz (\cite[Theorem 2.6.4]{Levin}). In fact, there is a bijection between the difference subvarieties of $\A^n_k$ and the perfect $\s$-ideals of $k\{y_1,\ldots,y_n\}$. Note however, that in this setup solutions are restricted to be solutions in $\s$-field extensions of $k$. Allowing solutions in more general \ks-algebras, such as rings of sequences, leads to a different kind of Nullstellensatz. (See \cite{PogudinScanlonWibmer:SolvingDifferenceEquationsInSequences}.)
Any perfect $\s$-ideal $I$ of $k\{y_1,\ldots,y_n\}$ can be written uniquely as an irredundant intersection $I=\p_1\cap\ldots\cap\p_m$ of prime reflexive $\s$-ideals (\cite[Theorem 2.5.7]{Levin}).

\begin{cor} \label{cor: sdim for perfect}
	Let $I\subseteq k\{y_1,\ldots,y_n\}$ be a perfect $\s$-ideal, written as an irredundant intersection $I=\p_1\cap\ldots\cap\p_m$ of prime reflexive $\s$-ideals. Then $\sdim(I)$ is the maximum (over $1\leq j\leq m$) of the $\s$-transcendence degrees of the fields of fractions of $k\{y_1,\ldots,y_n\}/\p_j$. In particular, for a reflexive prime $\s$-ideal $\p$, $\sdim(\p)$ equals the $\s$-transcendence degree of the field of fractions of $k\{y_1,\ldots,y_n\}/\p$. 
\end{cor}
\begin{proof}
%	For $i\geq 0$ we set $k\{y\}[i]=k[y_1,\ldots,y_n,\ldots,\s^i(y_1),\ldots,\s^i(y_n)]$ and for a $\s$-ideal $J$ of $k\{y_1,\ldots,y_n\}$ we set $J[i]=J\cap k\{y\}[i]$.
	% Since the intersection $I=\p_1\cap\ldots\cap\p_m$ is irredundant, i.e., $\p_j\nsubseteq \p_\ell$ for $j\neq \ell$ it follows that for $i\gg0$ the intersection
	With notation as in Remark \ref{rem: cut at i} we have $I[i]=\p_1[i]\cap\ldots\cap\p_m[i]$ for $i\geq 0$  and it follows that
	$$d_i=\dim(k\{y\}[i]/I[i])=\max\{\dim (k\{y\}[i]/\p_j[i]) |\ 1\leq j\leq m\}.$$
	As in the proof of Proposition \ref{prop: compare stredeg}, for every $1\leq j\leq m$, there exist $d(\p_j), e(\p_j)\in \N$ such that $$d_i(\p_j)=\dim(k\{y\}[i]/\p_j[i])=d(\p_j)(i+1)+e(\p_j)$$
	for $i\gg0$. Thus, if $j_0\in\{1,\ldots,m\}$ is such that $d(\p_{j_0})$ is maximal and $e(\p_{j_0})$ is maximal among all $e(\p_j)$ with $d(\p_j)$ maximal, then
	$d_i=d(\p_{j_0})(i+1)+e(\p_{j_0})$ for $i\gg 0$. It follows that
	$$\sdim(I)=\lim_{i\to\infty}\frac{d_i}{i+1}=d(\p_{j_0}).$$
	Since $d(\p_{j})$ agrees with the $\s$-transcendence degree of the field of fractions of $k\{y_1,\ldots,y_n\}/\p_j$ over $k$ the claim follows.
\end{proof}

We next compare our notion of $\s$-dimension with a difference analog of the Krull dimension. Let us first explain how the idea of the definition of the Krull dimension can be adapted to difference algebra. (Cf. \cite[Definition 4.6.1]{Levin} or \cite[Section 7.2]{KontradievaLevinMikhalev:DifferentialAndDiffereneDimensionPolynomials}.) Since the $\s$-polynomial ring $k\{y_1\}$ in one $\s$-variable contains infinite descending chains of prime $\s$-ideals one cannot simply take the maximal length of chains of prime $\s$-ideals as the definition. Instead one has to work with chains of chains: Let $R$ be a finitely $\s$-generated \ks-algebra. The largest integer $d\geq 0$ such that there exists a chain of infinite chains of prime $\s$-ideals of $R$ of the form
\begin{equation} \label{eq:chain}
\p_0\supsetneq\p_0^1\supsetneq\p_0^2\supsetneq\ldots \supsetneq\p_1\supsetneq\p_1^1\supsetneq\p_1^2\supsetneq\ldots \supsetneq\p_2\supsetneq \ldots\supsetneq \p_{d-1}\supsetneq\p_{d-1}^1\supsetneq\p_{d-1}^2\supsetneq\ldots\supsetneq\p_d
\end{equation}
is called the \emph{difference Krull dimension} of $R$ and denoted by $\dim_U(R)$. By definition $\dim_U(R)=0$ if $R$ has no (or only finitely many) prime $\s$-ideals. The existence of a maximal $d$ follows from the proof of Proposition \ref{prop:compare with Krull} below.
% We note that our definition of $\dim_U(R)$ agrees with the definition in \cite[Section 7.2]{KontradievaLevinMikhalev:DifferentialAndDiffereneDimensionPolynomials} only if $\operatorname{type}_U(R)=1$. In our situation $\operatorname{type}_U(R)\in\{0,1\}$ and we have chosen our definition to have a more complete comparison with $\sdim(R)$:

\begin{prop} \label{prop:compare with Krull}
Let $R$ be a finitely $\s$-generated \ks-algebra. Then $$\dim_U(R)\leq\sdim(R).$$
\end{prop}	
\begin{proof}
Let $A\subseteq R$ be finite such that $R=k\{A\}$. For a prime $\s$-ideal $\p$ of $R$ let $\overline A$ denote the image of $A$ in $R/\p$ and consider the sequence $(d_i)_{i\geq 0}$ defined by $d_i=\dim(k[\overline{A},\ldots,\s^{i}(\overline{A})])$. According to \cite[Lemma and Definition 4.21]{Hrushovski:elementarytheoryoffrobenius} (cf. \cite[Theorem 5.1.1]{Wibmer:AlgebraicDifferenceEquations}) there exist $d(\p),e(\p)\in\N$ such that $d_i=d(\p)(i+1)+e(\p)$ for $i\gg 0$. So the polynomial $\omega_\p(t)=d(\p)(t+1)+e(\p)$ satisfies $\omega_\p(i)=d_i$ for $i\gg 0$.

We define a total order on the set of polynomials of the form $d(t+1)+e$ with $d,e\in\N$ by  
$d(t+1)+e\leq d'(t+1)+e'$ if $d(i+1)+e\leq d'(i+1)+e'$ for $i\gg 0$. This is a well-order since it corresponds to the lexicographic order on pairs $(d,e)$. If $\p\supseteq\q$ are prime $\s$-ideals of $R$, then $\omega_\p(t)\leq\omega_\q(t)$.  Moreover, $\omega_\p(t)<\omega_\q(t)$ if $\p\supsetneq\q$.
So an infinite descending chain $\p\supsetneq\p^1\supsetneq \p^2\supsetneq\dots\supsetneq\q$ of prime $\s$-ideals in $R$ gives rise to an infinite ascending chain $\omega_\p(t)<\omega_{\p^1}(t)<\omega_{\p^2}(t)<\ldots <\omega_\q(t)$ of polynomials. But in such a chain we necessarily have $d(\p)<d(\q)$. Thus for a descending chain of infinite chains of prime $\s$-ideals as in equation (\ref{eq:chain}) we have $d({\p_d})\geq d$. So $d(\p_d)\geq\dim_U(R)$. 

As $\sdim(R)\geq \sdim(R/\p_d)=d(\p_d)$ by Proposition \ref{prop: morphism and sdim} (ii), it follows that $\sdim(R)\geq \dim_U(R)$ as desired.
\end{proof}

\begin{rem}
	In the definition of the difference Krull dimension above we have used prime $\s$-ideals. A similar invariant $\dim_{U^*}(R)$ could be obtained by modifying the definition by only allowing reflexive prime $\s$-ideals. Then clearly, $\dim_{U^*}(R)\leq\dim_U(R)$ and therefore also $\dim_{U^*}(R)\leq \sdim(R)$.
\end{rem}

The following example shows that the inequality from Proposition \ref{prop:compare with Krull} can be strict, even if $\sdim(R)$ is an integer.

\begin{ex}
	Consider $S=k\times k$ as a \ks-algebra via $\s(a,b)=(\s(b),\s(a))$ and $k\to S,\ \lambda\mapsto (\lambda,\lambda)$. Let $R=S\{y\}$ denote the univariate $\s$-polynomial ring over $S$. We first show that $R$ has no prime $\s$-ideals and so $\dim_U(R)=0$.
	
	Suppose $\p$ is a prime $\s$-ideal of $R$. Let $e_1=(1,0)\in S$ and $e_2=(0,1)\in S$. Since $e_1e_2=0\in\p$, we have $e_1\in\p$ or $e_2\in \p$. Assume (without loss of generality) that $e_1\in \p$. Since $\p$ is a $\s$-ideal, also $\s(e_1)=e_2\in \p$. But then $1=e_1+e+2\in \p$; a contradiction.
	
	To see that $\sdim(R)=1$, we  choose $A=\{e_1,e_2, y\}$. Then $\dim(k[A,\ldots,\s^i(A)])=i+1$ for all $i\in \N$ because $\{y_1,\ldots,\s^i(y)\}\subseteq A\cup\ldots\s^i(A)$ is an algebraically independent subset of maximal cardinality (Proposition \ref{prop: dim and algebraic independence}).
\end{ex}

%\begin{ex}\label{ex:sdim1}
%Let $y$ denote a $\s$-variable over $k$ and let $f=y\s^r(y)\s^{2r}(y)\cdots\s^{sr}(y)$ for some $r,s\geq1$. Then
%$$\sdim(k\{y\}/[f])=1-\frac{1}{s+1}.$$
%\end{ex}
%\begin{proof}
% ??
%\end{proof}
%
%\begin{ex} 
%Let $y$ denote a $\s$-variable over $k$ and let
%$$f=y\s(y)\cdots\s^{r-1}(y)\s^{2r}(y)\cdots\s^{3r-1}(y)\s^{4r}(y)\cdots\s^{5r-1}(y)\cdots\s^{2(s-1)r}(y)\cdots\s^{(2s-1)r-1}(y)$$
%??oder besser??
%$$f=\prod_{i=0}^{s-1}\s^{2ir}(y)\s^{2ir+1}(y)\cdots\s^{(2i+1)r-1}(y)$$
%for $r,s\geq 1$. Then
%$$\sdim(k\{y\}/[f])=1-\frac{1}{sr}.$$
%\end{ex}
%\begin{proof}
% ??
%\end{proof}
%
%
%
%Let $f=\s^{\alpha_1}(y)\cdots \s^{\alpha_m}(y)$ with $0\leq\alpha_1<\ldots<\alpha_m$. In the above two examples it holds that
%\begin{equation}\label{eq:wish}\sdim(k\{y\}/[f])=1-\frac{1}{m}.\end{equation} The following example shows that this formula is not true in general.
%
%\begin{ex}
%Let $y$ denote a $\s$-variable over $k$ and let $f=y\s(y)\s^3(y)$. Then
%$$\sdim(k\{y\}/[f])=\frac{3}{5}.$$
%\end{ex}
%\begin{proof}
% ??
%\end{proof}
%
%While equation \ref{eq:wish} fails in general, the number $1-\frac{1}{m}$ at least provides an upper bound:
%\begin{prop}
%Let $y$ denote a $\s$-variable over $k$ and let $f=\s^{\alpha_1}(y)^{\beta_1}\cdots\s^{\alpha_m}(y)^{\beta_m}$ with $m\geq2$, $0\leq\alpha_1<\ldots<\alpha_m$ and $\beta_1,\ldots,\beta_m\geq 1$. Then
%$$\frac{1}{2}\leq\sdim(k\{y\}/[f])\leq1-\frac{1}{m}.$$
%\end{prop}
%\begin{proof}
% ??
%\end{proof}
%
%
%
%
%

\section{Covering density and the dimension of difference monomials} \label{sec:Covering Density and the dimension of difference monomials}

In this section we determine the $\s$-dimension of a univariate $\s$-monomial $\s^{\alpha_1}(y)^{\beta_1}\ldots\s^{\alpha_n}(y)^{\beta_n}$. It turns out that this $\s$-dimension is essentially given by the \emph{covering density} of $\{\alpha_1,\ldots,\alpha_n\}$.

There is a vast body of literature on covering, packing and tiling problems. We refer the interested reader to \cite{BollobasJansonRiordan:OnCoveringByTranslatesOfaSet} and the references given there. In rather general terms the covering problem can be formulated as follows: Given an additive group $G$ and a subset $E$ of $G$, find a ``minimal'' subset $E'$ of $G$ such that $E+E'=\{e+e'|\ e\in E,\ e'\in E'\}$ equals $G$. Such an $E'$ is often called a \emph{complement} of $E$. 
It is instructive to think of $E+E'$ as a union of translates $E+e'$ of $E$. The question then becomes, ``how many'' translates of $E$ are needed to cover $G$? To give a precise meaning to ``minimal'' and ``how many'' one usually assumes that $G$ is equipped with some measure or density. A well studied special case is $G=\R^n$ and $E$ a ball or convex body. For our purpose we are interested in the case $G=\Z$ and $E$ a finite set, studied e.g., in \cite[Section 5]{BollobasJansonRiordan:OnCoveringByTranslatesOfaSet}, \cite{Newman:ComplememtsOfFiniteSetsOfIntegers},\cite{Weinstein:SomeCoveringAndPackingResultsInNumberTheory},\cite{Tuller:thesis},\cite{Schmidt:ComplementarySetsOfFiniteSets},\cite{SchmidtTuller:CoveringAndPackingI},\cite{SchmidtTuller:CoveringAndPackingII}.

For a finite subset $E$ of $\Z$, the \emph{covering density} $\c(E)$ of $E$ can be defined as 
$$\c(E)=\inf_{E'}\d(E'),$$
where $\d(E')=\lim_{i\to\infty}\frac{|E'\cap[-i,i]|}{2i}$ is the density of $E'$ and
the infimum is taken over all complements of $E$ for which the density exists. We note that the covering density is called the \emph{codensity} in \cite{Newman:ComplememtsOfFiniteSetsOfIntegers} and the \emph{minimal covering frequency} in \cite{SchmidtTuller:CoveringAndPackingI,SchmidtTuller:CoveringAndPackingII}. We are using the nomenclature from \cite{BollobasJansonRiordan:OnCoveringByTranslatesOfaSet}. As pointed out in 
\cite[Section 5]{BollobasJansonRiordan:OnCoveringByTranslatesOfaSet}, there is an equivalent definition of $c(E)$, which we will use: For $i\geq 1$ let $\tau(E,i)$ be the smallest number of translates of $E$ that cover $\{1,\ldots,i\}$, i.e.,$$\tau(E,i)=\min\{|E'|\ |\ E+E'\supseteq\{1,\ldots,i\} \}.$$
Then $\c(E)=\lim_{i\to\infty}\frac{\tau(E,i)}{i}$.

\begin{theo} \label{theo: covering density and smonomial}
	The $\s$-dimension of a univariate $\s$-monomial $\s^{\alpha_1}(y)^{\beta_1}\ldots\s^{\alpha_n}(y)^{\beta_n}$ with $0\leq\alpha_1<\alpha_2<\ldots<\alpha_n$ and $\beta_1,\ldots,\beta_n\geq 1$ is $1-\c(E)$, where $c(E)$ is the covering density of $E=\{\alpha_1,\ldots,\alpha_n\}$.
\end{theo}
\begin{proof}
	We first observe that $\sdim(\s^{\alpha_1}(y)^{\beta_1}\ldots\s^{\alpha_n}(y)^{\beta_n})=\sdim(\s^{\alpha_1}(y)\ldots\s^{\alpha_n}(y))$ by Corollary \ref{cor: morphism and sdim} (ii) and Corollary \ref{cor: radical}, where we use that $$[\s^{\alpha_1}(y)^{\beta_1}\ldots\s^{\alpha_n}(y)^{\beta_n}]\subseteq [\s^{\alpha_1}(y)\ldots\s^{\alpha_n}(y)]\subseteq \sqrt{[\s^{\alpha_1}(y)^{\beta_1}\ldots\s^{\alpha_n}(y)^{\beta_n}]}.$$
So it remains to show that $\sdim(f)=1-\c(E)$ for $f=\s^{\alpha_1}(y)\ldots\s^{\alpha_n}(y)$.

As in Remark \ref{rem: cut at i}, we set $k\{y\}[i]=k[y,\ldots,\s^i(y)]$ and $[f][i]=[f]\cap k\{y\}[i]$ for $i\geq 0$. Then $\sdim(f)=\lim_{i\to\infty}\frac{d_i}{i+1}$, where $d_i=\dim(k\{y\}[i]/[f][i])$.
%
%For $i\geq 1$ and a $\s$-ideal $\ida$ of $k\{y\}$ we set $\ida[i]=\ida\cap k[y,\s(y),\ldots,\s^{i-1}(y)]$. Then $\sdim(\ida)=\lim_{i\to\infty}\frac{\dim(\ida[i])}{i}$.

For an arbitrary $F\subseteq k\{y\}$, it is non-trivial to determine $[F][i]$. However, in our situation, since we are only dealing with monomial ideals, we see that $$[f][i]=[f,\s(f),\ldots,\s^{i-\alpha_n}(f)]\subseteq k\{y\}[i]$$ for $i\geq\alpha_n$.
To determine the dimension of this monomial ideal, let us recall (\cite[Chapter 9, \S 1, Prop. 3]{CoxLittleOShea:IdealsVarietiesAndAlgorithms}) how to determine the dimension of a monomial ideal $M=(f_1,\ldots,f_r)\subseteq k[y_1,\ldots,y_m]$ in general, where $f_j=\prod_{l\in S_j}y_l$ and $S_1,\ldots,S_r\subseteq \{1,\ldots,m\}$. The solution set of $M$ is a finite union of coordinate subspaces and to find the dimension of $k[y_1,\ldots,y_m]/M$, it suffices to find the coordinate subspace of the largest dimension, which is given by $$m-\min\{|T|\ |\ T\subseteq\{1,\ldots,m\},\ T\cap S_j\neq\emptyset \text{ for } j=1,\ldots,r\}.$$
	Therefore $$\dim(k\{y\}[i]/[f][i])=i+1-\min\{|T|\ |\ T\subseteq\{0,\ldots,i\},\ T\cap (E+j)\neq\emptyset \text{ for } j=0,\ldots,i-\alpha_n\}.$$	
But for $T\subseteq\{0,\ldots,i\}$, we have $T\cap (E+j)\neq\emptyset \text{ for } j=0,\ldots,i-\alpha_n$ if and only if $\{0,\ldots,i-\alpha_n\}\subseteq \cup_{t\in T}(-E+t)$, where $-E=\{-e|\ e\in E\}$. Thus
	\begin{align*}
	&\min\{|T|\ |\ T\subseteq\{0,\ldots,i\},\ T\cap (E+j)\neq\emptyset \text{ for } j=0,\ldots,i-\alpha_n\} \\
	&=\min\{|T|\ |\ T\subseteq\{0,\ldots,i\},\ \{0,\ldots,i-\alpha_n\}\subseteq -E+T \}\\
	&=\min\{|T|\ |\ T\subseteq \Z,\  \{0,\ldots,i-\alpha_n\}\subseteq -E+T\}\\
	&=\min\{|T|\ |\ T\subseteq \Z,\  \{1,\ldots,i-\alpha_n+1\}\subseteq -E+T\}\\
	&=\tau(-E,i-\alpha_n+1)
	\end{align*}
and so, $d_i=i+1-\tau(-E,i-\alpha_n+1)$.
Consequently,
\begin{align*}
\sdim(f)&=\lim_{i\to\infty}\frac{d_i}{i+1}=1-\lim_{i\to\infty}\frac{\tau(-E,i-\alpha_n+1)}{i+1}=\\ &=1-\lim_{i\to\infty}\frac{\tau(-E,i-\alpha_n+1)}{i-\alpha_n+1}\left(\frac{i-\alpha_n+1}{i+1}\right)=\\
&=1-\c(-E)\cdot 1.
\end{align*}
Since $\c(-E)=\c(E)$ (\cite[Lemma 2.8]{Tuller:thesis}) the claim follows.	
\end{proof}

\begin{ex} \label{ex: sdim of monomial}
	The covering density of a one-element set is $1$ and the covering density $c(E)$ of a finite subset $E$ of $\Z$ with at least two elements satisfies $\frac{1}{|E|}\leq c(E)\leq \frac{1}{2}$ (\cite[Lemma 2.9]{Tuller:thesis}). Moreover, $c(E)$ is rational (\cite[Theorem 2.13]{Tuller:thesis} or \cite[Theorem~5.1]{BollobasJansonRiordan:OnCoveringByTranslatesOfaSet}).
	
	Thus the $\s$-dimension of a $\s$-monomial $\s^{\alpha_1}(y)^{\beta_1}\ldots\s^{\alpha_n}(y)^{\beta_n}$ is $0$ if $n=1$ and otherwise it is a rational number between $\frac{1}{2}$ and $1-\frac{1}{n}$.
\end{ex}

\section{Values of the difference dimension} \label{sec:Values of the difference dimension}

As seen in Example \ref{ex: sdim of monomial} above, the $\s$-dimension of a system of algebraic difference equations need not be an integer. This raises two questions:
\begin{itemize}
\item When is the $\s$-dimension an integer?
\item What values can the $\s$-dimension take?
\end{itemize}

Concerning the first question, we add to the already known cases, the case of a finitely $\s$-generated \ks-Hopf algebra. We do not fully answer the second question but we reduce it to a purely combinatorial problem. This reduction shows in particular, that the answer does not depend on the base $\s$-field $k$.

We have already seen that the $\s$-dimension of $R=k\{y_1,\ldots,y_n\}/I$ is an integer in all of the following cases:
\begin{itemize}
	\item $R$ is an integral domain, i.e., $I$ is a prime $\s$-ideal (Proposition \ref{prop: compare stredeg}).
	\item $I=[F]$ for some $F\subseteq k[y_1,\ldots,y_n]$ (Proposition \ref{prop: sdim for algebraic}).
	\item $I$ is a perfect $\s$-ideal (Corollary \ref{cor: sdim for perfect}).
\end{itemize}

The following theorem shows that the $\s$-dimension of a finitely $\s$-generated \ks-Hopf algebra is also always an integer. This result was already alluded to in \cite[Remark~A.30]{DiVizioHardouinWibmer:DifferenceGaloisofDifferential}. Hopf algebras are important in algebraic geometry because they are the coordinate rings of affine group schemes (\cite[Section~1.4]{Waterhouse:IntrotoAffineGroupSchemes}). Hopf algebras over a field $k$ that are finitely generated as $k$-algebras correspond to affine group schemes of finite type over $k$, i.e., affine (sometimes also called linear) algebraic groups. A similar duality exists in difference algebraic geometry: \ks-Hopf algebras that are finitely $\s$-generated as \ks-algebras correspond to affine difference algebraic groups. See \cite[Appendix~A]{DiVizioHardouinWibmer:DifferenceGaloisofDifferential},  \cite{Wibmer:FinitenessPropertiesOfAffineDifferenceAlgebraicGroups} and \cite{Wibme:AlmostSimpleAffineDiffferenceAlgebraicGroups} for more background of affine difference algebraic groups.

%
%\ks-Hopf algebras are important in difference algebra because they are the coordinate rings of affine difference algebraic groups.
\begin{theo}[{\cite[Theorem 3.7]{Wibmer:FinitenessPropertiesOfAffineDifferenceAlgebraicGroups}}]
	Let $R$ be a finitely $\s$-generated $k$-$\s$-algebra. Assume that $R$ can be equipped with the structure of a $k$-$\s$-Hopf algebra, i.e., there exist morphisms of $k$-$\s$-algebras $\Delta\colon R\to R\otimes_k R$, $S\colon R\to R$ and $\varepsilon\colon R\to k$ that turn $R$ into a Hopf algebra. Then $\sdim(R)$ is an integer.
\end{theo}
\begin{proof}
	In \cite[Theorem 3.7]{Wibmer:FinitenessPropertiesOfAffineDifferenceAlgebraicGroups} it is shown that there exists a finite subset $A$ of $R$ such that $k\{A\}=R$, $k[A]$ is a Hopf-subalgebra of $R$ and 
	$\dim(k[A,\ldots,\s^i(A)])=d(i+1)+e$ for some $d,e\in \N$ and $i\gg 0$. So $\sdim(R)=d\in\N$.
\end{proof}

We next address the question, which non-negative real numbers $d$ are of the form $d=\sdim(F)$ for some $F\subseteq k\{y_1,\ldots,y_n\}$? As a first step, we show that one can reduce to the case that $F$ consists of $\s$-monomials. Then, we will further reduce to the case of monomial $\s$-ideals generated by squarefree $\s$-monomials.

A \emph{$\s$-monomial} in the $\s$-variables $y_1,\ldots,y_n$ is a monomial in the variables $\s^i(y_j)$, $i\in\N$, $j\in\{1,\ldots,n\}$. A $\s$-ideal $M$ of $k\{y_1,\ldots,y_n\}$ is a \emph{monomial $\s$-ideal} if it is of the form $M=[F]$ for some set $F\subseteq k\{y_1,\ldots,y_n\}$ of $\s$-monomials.
% generated by $\s$-monomials. The following proposition reduces the question which values the $\s$-dimension can take to the study of the $\s$-dimension of monomial $\s$-ideals.

\begin{lemma} \label{lem: groebner}
	For any $F\subseteq k\{y_1,\ldots,y_n\}$ there exists a monomial $\s$-ideal $M$ of $k\{y_1,\ldots,y_n\}$ with $\sdim(F)=\sdim(M)$.
\end{lemma}
\begin{proof}
	For the proof we will use some notions (orderings and leading monomials) from the theory of difference Gr\"{o}bner bases (\cite{LaScala:GroebnerBasesAndGradingsForPartialDifferenceIdeals,GerdtLaScala:NoetherianQuotientsOfTheAlgebraOfPartialDifferencePolynomialsAndGroebnerBases}). We fix a total order $\leq$ on the set of all $\s$-monomials in $y_1,\ldots,y_n$. Indeed, let us be concrete and choose $\leq$ as the lexicographic order with $$y_1<y_2<\ldots<y_n<\s(y_1)<\s(y_2)<\ldots<\s(y_n)<\s^2(y_1)<\ldots.$$ Then $\leq$ satisfies the following properties:
	\begin{enumerate}
		\item $\leq$ is a well-order, i.e., every descending chain of $\s$-monomials is finite.
		\item $1\leq f$ for every $\s$-monomial $f$.
		\item If $f\leq g$, then $hf\leq hg$ for $\s$-monomials $f,g,h$.		
		\item If $f\leq g$, then $\s(f)\leq\s(g)$ for $\s$-monomials $f,g$.
		\item If $\ord(f)<\ord(g)$, then $f<g$ for $\s$-monomials $f,g$.
	\end{enumerate}
Recall that the order $\ord(f)$ of a $\s$-polynomial $f$ is the largest power of $\s$ that occurs in $f$.
Let us write a non-zero $\s$-polynomial $f\in k\{y_1,\ldots,y_n\}$ as $f=\sum_{j=1}^m c_jf_j$ for coefficients $c_j\in k\smallsetminus\{0\}$ and distinct $\s$-monomials $f_j$. The \emph{leading monomial} $\lm(f)$ of $f$ is the largest $f_j$. For $f=0$, we set $\lm(f)=0$.
For a $\s$-ideal $I$ of $k\{y_1,\ldots,y_n\}$, we set
$$\lm(I)=(\lm(f)|\ f\in I)\subseteq k\{y_1,\ldots,y_n\}.$$  Thanks to (iv) above, we see that $\lm(I)$ is a $\s$-ideal.

Define $I=[F]$ and $M=\lm(I)$. Then $M$ is a monomial $\s$-ideal and we claim that $\sdim(I)=\sdim(M)$.

With notation as in Remark \ref{rem: cut at i}, we have for $i\geq 0$, thanks to (v), that
$\lm(I[i])=\lm(I)[i]$,
where $\lm(I[i])$ is the ideal of leading monomials of $I[i]\subseteq k[y_1,\ldots,y_n,\ldots,\s^{i}(y_1),\ldots,\s^{i}(y_n)]$ with respect to the lexicographic order with $y_1<y_2<\ldots<\s^{i}(y_n)$.
%
%As in the proof of Theorem \ref{theo: covering density and smonomial} we set $\idb[i]=\idb\cap k[y_1,\ldots,y_n,\ldots,\s^{i-1}(y_1),\ldots,\s^{i-1}(y_n)]$ for $i\geq 1$ and $\idb\subseteq k\{y_1,\ldots,y\}$ a $\s$-ideal. It follows from (v) that $\lm(\idb[i])=\lm(\idb)[i]$, where $\lm(\idb[i])$ is the ideal of leading monomials of $\idb[i]\subseteq k[y_1,\ldots,y_n,\ldots,\s^{i-1}(y_1),\ldots,\s^{i-1}(y_n)]$ with respect to the lexicographic order with $y_1<y_2<\ldots<\s^{i-1}(y_n)$. 
The dimension of an ideal in a polynomial ring over a field agrees with the dimension of its ideal of leading monomials (\cite[Corollary 7.5.5]{GreuelPfister:SingularIntroductionToCommutativeAlgebra}). Thus
$$\dim(k\{y\}[i]/I[i])=\dim(k\{y\}[i]/\lm(I[i]))=\dim(k\{y\}[i]/\lm(I)[i])=\dim(k\{y\}[i]/M[i])$$ and $\sdim(I)=\sdim(M)$ as desired. 
\end{proof}

It remains to determine the possible $\s$-dimensions of monomial $\s$-ideals. As we will see, this can be reduced to a purely combinatorial problem, which we now describe.
 
Define $\s\colon \N\times\{1,\ldots,n\}\to \N\times\{1,\ldots,n\}$ by $\s(i,j)=(i+1,j)$. For a finite subset $S$ of $\N\times\{1,\ldots,n\}$ we set $\ord(S)=\max\{i|\ \exists\ j : (i,j)\in S\}$. Let $\mathcal{S}$ be a set of non-empty finite subsets of $\N\times\{1,\ldots,n\}$. For $i\geq 0$ we define
$$\tau(\mathcal{S},i)=\min\{|T|\ |\ T\subseteq \N\times\{1,\ldots,n\},\ T\cap\s^\ell(S)\neq \emptyset, \ \forall \ S\in\mathcal{S},\ 0\leq\ell\leq i-\ord(S) \}.$$
In other words, if $[\mathcal{S}]=\{\s^\ell(S)|\ S\in\mathcal{S},\ \ell\in\N\}$ and $$[\mathcal{S}][i]=\{S\in [\mathcal{S}]|\ S\subseteq \{0,\ldots,i\}\times \{1,\ldots,n\}\},$$
then $$\tau(\mathcal{S},i)=\min\{|T|\ |\ T\subseteq \N\times\{1,\ldots,n\},\ T\cap S\neq \emptyset, \ \forall \ S\in [\mathcal{S}][i]\}.$$

It follows from the proof of the following lemma (and Theorem \ref{theo:sdimdef}) that $C(\mathcal{S})=\lim_{i\to\infty}\frac{\tau(\mathcal{S},i)}{i+1}$ exists.
Since $T=\{0,\ldots,i\}\times\{1,\ldots,n\}$ intersects every non-empty subset of $\{0,\ldots,i\}\times\{1,\ldots,n\}$, we have $\tau(\mathcal{S},i)\leq (i+1)n$ and therefore $0\leq C(\mathcal{S})\leq n$. We set $\sdim(\mathcal{S})=n-C(\mathcal{S})$.

%
%
% A subset $T$ of $\N\times\{1,\ldots,n\}$ is a \emph{support} of $\mathcal{S}$ if $T\cap \s^i(S)\neq\emptyset$ for every $i\geq 0$ and $S\in\mathcal{S}$. Set $C(\mathcal{S})=\inf_Td(T)$, where the infimum is taken over all supports $T$ of $\mathcal{S}$ that have a defined density. Here the density $d(T)$ of a set $T\subseteq\N\times\{1,\ldots,n\}$ is defined as $\lim_{i\to\infty}\frac{|T\cap(\{0,\ldots,i-1\}\times\{1,\ldots,n\})|}{in}$ if this limit exists. We call $\sdim(\mathcal{S})=n(1-C(\mathcal{S}))$ the $\s$-dimension of $\mathcal{S}$.

For a finite subset $S$ of $\N\times\{1,\ldots,n\}$ we set $y^S=\prod_{(i,j)\in S}\s^i(y_j)$.  Furthermore we define $M(\mathcal{S})=[\{y^S|\ S\in\mathcal{S}\} ]\subseteq k\{y_1,\ldots,y_n\}$.
The proof of the following lemma, generalizes some aspects of the proof of Theorem~\ref{theo: covering density and smonomial}.

\begin{lemma} \label{lem: sdim of smonomial ideal}
	Let $\mathcal{S}$ be a set of non-empty finite subsets of $\N\times\{1,\ldots,n\}$. Then $\sdim(M(\mathcal{S}))=\sdim(\mathcal{S})$.
\end{lemma}
\begin{proof}
	Using the notation of Remark \ref{rem: cut at i}, we have
	$$M(\mathcal{S})[i]=\left(\s^\ell(y^S)|\ S\in\mathcal{S},\ 0\leq \ell\leq i-\ord(S)\right)\subseteq k\{y\}[i]$$
	for every $i\geq 0$. Using the description of the dimension of monomial ideals in a polynomial ring as in the proof of Theorem \ref{theo: covering density and smonomial} (cf. \cite[Chapter 9, \S 1, Prop. 3]{CoxLittleOShea:IdealsVarietiesAndAlgorithms}), we see that
	$\dim(k\{y\}[i]/M(\mathcal{S})[i])=n(i+1)-e_i$ where 
	\begin{align*}e_i & =\min\{|T|\ |\ T\subseteq \{0,\ldots,i\}\times\{1,\ldots,n\},\ T\cap\s^\ell(S)\neq \emptyset, \ \forall \ S\in\mathcal{S},\ 0\leq\ell\leq i-\ord(S) \} \\
	& =\min\{|T|\ |\ T\subseteq \N\times\{1,\ldots,n\},\ T\cap\s^\ell(S)\neq \emptyset, \ \forall \ S\in\mathcal{S},\ 0\leq\ell\leq i-\ord(S)\}\\
	& =\tau(\mathcal{S},i).
	\end{align*}
	Hence $$\sdim(M(\mathcal{S}))=\lim_{i\to\infty}\dim(k\{y\}[i]/M(\mathcal{S})[i])=n-\lim_{i\to\infty}\frac{\tau(\mathcal{S},i)}{i+1}=\sdim(\mathcal{S}).$$
\end{proof}

The following theorem gives a combinatorial description of all numbers that occur as the $\s$-dimension of a finitely $\s$-generated \ks-algebra (equivalently of a system of algebraic difference equations).

\begin{theo}
	Let $d\geq 0$ be a real number. Then $d=\sdim(F)$ for some $F\subseteq k\{y_1,\ldots,y_n\}$ if and only if $d=\sdim(\mathcal{S})$ for some set $\mathcal{S}$ of non-empty finite subsets of $\N\times\{1,\ldots,n\}$.
\end{theo}
\begin{proof}
	If $d=\sdim(\mathcal{S})$, then $d=\sdim(F)$ for $F=M(\mathcal{S})$ by Lemma \ref{lem: sdim of smonomial ideal}. 
	
	Conversely, assume that $d=\sdim(F)$ for some $F\subseteq k\{y_1,\ldots,y_n\}$. By Lemma~\ref{lem: groebner} we can assume without loss of generality that $F=M$ is a monomial $\s$-ideal. 
	Let $E\subseteq k\{y_1,\ldots,y_n\}$ be a set of $\s$-monomials such that $M=[E]\subseteq k\{y_1,\ldots,y_n\}$. 
	
	Let us refer to a $\s$-monomial as square-free if it is square-free as a monomial in the variables $\s^i(y_j)$. The square-free part of a $\s$-monomial is defined in a similar spirit, i.e., by replacing all non-zero exponents with $1$'s. Let $E'\subseteq k\{y_1,\ldots,y_n\}$  be the set of all square-free parts of all $\s$-monomials in $E$. Then $$[E]\subseteq[E']\subseteq\sqrt{[E]}.$$
	It thus follows from Corollary \ref{cor: morphism and sdim} (ii) and Corollary \ref{cor: radical} that $\sdim([E])=\sdim([E'])$.
	To specify a (non-constant) square-free $\s$-monomial is equivalent to specifying a (non-empty) finite subset $S$ of $\N\times\{1,\ldots,n\}$.
	Thus $[E']=M(\mathcal{S})$ for some set $\mathcal{S}$ of finite non-empty subsets of $\N\times\{1,\ldots,n\}$. In summary, $$\sdim(F)=\sdim([E])=\sdim([E'])=\sdim(M(\mathcal{S}))=\sdim(\mathcal{S}),$$
	by Lemma \ref{lem: sdim of smonomial ideal}.	
\end{proof}

\bibliographystyle{alpha}
\bibliography{bibdata}

\end{document}